\documentclass[reqno,11pt]{imsart}
\usepackage[utf8]{inputenc}
\usepackage{amsmath,amsfonts,amssymb,amsthm,epsfig,epstopdf,titling,url,array}
\usepackage{blindtext}
\usepackage{multirow}
\usepackage{algorithm}
\usepackage{algpseudocode}
\usepackage{caption}
\usepackage{subcaption}
\usepackage{graphicx}
\usepackage{verbatim}
\usepackage{xcolor}
\usepackage[numbers]{natbib}
\usepackage[capitalize,noabbrev]{cleveref}
\usepackage{mathtools}

\theoremstyle{plain}
\newtheorem{thm}{Theorem}[section]

\newtheorem{cor}[thm]{Corollary}

\theoremstyle{definition}
\newtheorem{defn}{Definition}[section]

\newtheorem{assumption}{Assumption}[section]

\theoremstyle{remark}
\newtheorem{rem}{Remark}

\usepackage{xcolor}

\def\ilsang{\textcolor{blue}}

\begin{document}

\title{Variational bagging: a robust  approach for Bayesian uncertainty quantification}

\runtitle{Variational bagging}

\author{Shitao Fan, Ilsang Ohn, David Dunson and Lizhen Lin}

\runauthor{Fan, Ohn, Dunson and Lin}
\date{}

\begin{abstract}
Variational Bayes methods are popular due to computational efficiency and adaptivity to diverse applications. In specifying the variational family, mean-field classes are commonly chosen, which enables efficient algorithms such as coordinate ascent variational inference (CAVI), but fails to capture parameter dependence and typically underestimates uncertainty. In this work, we introduce a variational bagging approach that integrates a bagging procedure with variational Bayes, resulting in a \emph{bagged variational posterior} for improved inference. We establish strong theoretical guarantees, including posterior contraction rates for general models and a Bernstein–von Mises (BVM)-type theorem that ensures valid uncertainty quantification. Notably, our results show that even when using a mean-field variational family, our approach can recover off-diagonal elements of the limiting covariance structure and, crucially, provide proper uncertainty quantification.  In addition, variational bagging is robust to model misspecification, with the covariance structures matching  that of the target covariance.  
We illustrate  our variational bagging
in numerical studies through  applications to parametric models,  finite mixture models, deep neural networks, and variational autoencoders (VAEs).

\end{abstract}
\maketitle

\section{Introduction}

Variational Bayes has emerged as a powerful framework for scalable Bayesian inference by approximating posterior distributions in complex and high-dimensional models with simpler variational families bypassing the need for MCMC sampling. There have been wide applications including for topic modeling \cite{vb-lda},  graphical models \cite{vb-graphical, vb-dag}, Bayesian nonparametric modeling \citep{vb-dpmixture}, and high-dimensional sparse models \cite{Ray03072022}. Variational Bayes also plays an important role in modern generative AI through the variational autoencoder (VAE) \citep{VAE-kingma-welling, chae2021likelihood}. 
In specifying the variational family, a commonly used approximation within this framework is the mean-field family, which simplifies computation by assuming independence among model parameters. Although this assumption facilitates efficient algorithms like coordinate ascent variational inference (CAVI), it can lead to poor uncertainty quantification by ignoring parameter dependencies and systematically underestimating posterior variance. Although there have been attempts to mitigate problems with mean-field approaches \cite{Giordano2017CovariancesRA, Vakili2022ImprovedCR, Katsevich2023OnTA, 10.1214/23-EJS2155}, most existing approaches are tailored to specific model classes, for example Gaussian or sparse Gaussian process models.

In this work, we propose \emph{variational bagging}, a method that combines variational Bayes with a bootstrap aggregation (bagging) procedure to produce a bagged variational posterior with improved inferential properties, especially uncertainty quantification. We will show that variational bagging, which yields a \emph{bagged variational posterior}, comes with strong theoretical guarantees, including posterior contraction rates for general models and a Bernstein–von Mises (BvM)-type result that ensures valid uncertainty quantification asymptotically. Remarkably, even when restricting to a mean-field variational family, our method can recover aspects of the full covariance structure, including off-diagonal elements.

Beyond the well-known undercoverage of variational Bayes, there is a broader concern about robustness to model misspecification. 
 Bayesian statistics, including variational Bayes,  is a model-based approach, 
 which comes with the implicit assumption that the likelihood is correctly specified. In reality, however, knowledge about the true model or model class is rarely given, and model misspecification is more often than not. For many model classes, the posterior distribution (or their variational approximations) and our corresponding inferences and predictions are sensitive to model misspecification. As sample size increases in the misspecified case, the posterior typically concentrates around the pseudo-true parameter value corresponding to the minimal Kullback-Leibler (KL) divergence from the true data-generating model. \citet{kleijn2012bernstein} provide a Bernstein von Mises (BvM) theory characterizing the limiting form of the posterior under misspecification, showing that in general Bayesian credible sets are not valid confidence sets when the model is misspecified even asymptotically. We will show that variational bagging is robust to model misspecification, and yield a covariance that matches the target covariance. 
 

`BayesBag' is a simple and general approach  to obtain a robustified posterior by averaging over posteriors defined conditionally on different bootstrap-replicated datasets \cite{buhlmann14, huggins2019robust}. This is an application of the  bootstrap aggregation (bagging) approach of \cite{breiman96}. Let $X_{1:M}^*=(X_1^*,\ldots,X_M^*)$ be a bootstrapped sample  of size $M$ from the  data $X=(X_1,\ldots,X_n)$ with sample size $n$.  Let $\pi(\theta|X^*)$ be the posterior distribution conditioned on the bootstrapped copy of the data.  The bagged posterior \cite{huggins2019robust} is defined as     \begin{align}
    \label{eq-intro-bag}
    \tilde{\pi}(\theta|X_{1:n})=\frac{1}{n^M}\sum_{X_{1:M}^*}\pi(\theta|X_{1:M}^*),
   \end{align}
which is defined over all possible bootstrap datasets. 
To reduce computational burden, we may use $B$ bootstrap samples, with $B \ll n^M$, to obtain
    \begin{align*}
    \tilde{\pi}(\theta|X_{1:n})\approx\frac{1}{B}\sum_{b=1}^B \pi\left(\theta|X_{(b)}^*\right),
    \end{align*} 
where each $X_{(b)}^*$ is a bootstrap sample of size $M$. \citet{huggins2019robust} studied asymptotic properties of the bagged posterior, including a BvM type theorem, and showed good predictive and uncertainty quantification (UQ) performance under model misspecification in simulation studies. 

One potential drawback with BayesBag is its computation cost in the typical case in which $\pi\left(\theta|X_{1:M}^*\right)$ is intractable. Running MCMC sampling for each bootstrap replicate is often infeasible. We propose a \emph{variational bagging approach} by first providing a variational approximation of each $\pi\left(\theta|X_{1:M}^*\right)$, the collection of which is then aggregated to produce the final \emph{bagged variational posterior distribution} for inference.  Compared to BayesBag,  variational bagging massively speeds up the computation while yielding new and interesting theoretical results
that rely on the unique properties of variational Bayes.

To summarize, we outline our theoretical results: 
\begin{itemize}
    \item[1.] First, although variational approximations are well known to underestimate posterior variance in many cases, we provide a BvM theorem showing that bootstrap aggregation not only accommodates model misspecification but also appropriately inflates the variance so that the bagged posterior has theoretical guarantees of accurate uncertainty quantification. Remarkably,  the limiting covariance structure (with simple adjustment) matches the target covariance even if the popular mean-field variational family is adopted. When variational inference is not conducted over the latent variable $Z$, the covariance  coincides with the asymptotic variance of the MLE under model
misspecification. 
\item[2.] Second, we provide theoretical results on posterior contraction rates of the resulting bagged VB posterior and conditions on when such rates are minimax optimal.  Our results encompass complex and nonparametric models and required the development of new technical tools. 

\item[3.] Finally, we show that bagged variational posterior distributions lack overconfident credible sets.
\end{itemize}

From the practical side, a major advantage of variational bagging is in providing a natural mechanism for better approximating complex multimodal posteriors by combining many local variational approximations from different bootstrapped data.  In many modern complex models, ranging from intricate mixture models to deep neural networks, MCMC algorithms  can fail to adequately explore the complex and multi-modal posterior landscape, while simple variational approximations without bagging only capture local features of this landscape.  The variational bagging procedure can naturally overcome the above issues and we demonstrate this in a rich variety of examples including finite mixture models, deep neural networks, and variational auto-encoders.



The paper is organized as follows. 
Section \ref{sec-prelim} introduces notation and background on MLE, Bayes and variational Bayes under model misspecification. In \cref{sec-vbag}, we describe the variational bagging algorithm in general and provide associated theory, while discussing practical aspects. \cref{sec-appmodel} shows applications of variational bagging to a variety of models, providing algorithmic details and theory support in these specific contexts.
Section \ref{sec-sim} provides an empirical evaluation and demonstration through a simulation study and real data analyses.

\section{Preliminaries}
\label{sec-prelim}
\subsection{Model setup and notation}

We consider a probability triple $(\mathcal{X}, \mathcal{F}, P_0)$ with corresponding density $p_0$. Let $X_{1:n} = (X_1,\ldots, X_n)$ be an i.i.d.\ sample from $P_0$. We model the data using a parametric family $\mathcal{P}_{\Theta} = \{ P_{\theta} : \theta \in \Theta \}$, where each $P_{\theta}$ has density $p(\cdot \mid \theta)$ indexed by a parameter $\theta$.

The true distribution $P_0$ is not assumed to belong to the parametric family $\mathcal{P}_{\Theta}$; that is, we explicitly allow for model misspecification. We denote by $\theta_0$ the pseudo-true parameter, defined as the minimizer of the Kullback--Leibler (KL) divergence:
\begin{align*}
    \theta_0 = \arg\min_{\theta \in \Theta} \textup{KL}(P_{\theta}, P_0),
\end{align*}
where $\textup{KL}(\cdot,\cdot)$ denotes the KL divergence
\[
    \textup{KL}(P_1, P_2) = \int \log\!\left(\frac{dP_1}{dP_2}\right) dP_1,
\]
whenever $P_1$ is absolutely continuous with respect to $P_2$. When the corresponding densities $p_1$ and $p_2$ exist, we write $\textup{KL}(p_1, p_2)$ for the induced KL divergence. When the model is correctly specified, we have $P_{\theta_0} = P_0$.

The maximum likelihood estimator (MLE),
\[
    \hat{\theta}_{\mathrm{mle}} = \arg\max_{\theta \in \Theta} \sum_{i=1}^n \log p(X_i \mid \theta),
\]
is asymptotically centered at $\theta_0$ and, under standard regularity conditions, satisfies
\begin{align*}
    \sqrt{n}\bigl(\hat{\theta}_{\mathrm{mle}} - \theta_0\bigr) 
    \stackrel{d}{\to} N\bigl(0,\, V(\theta_0)^{-1} D(\theta_0) V(\theta_0)^{-1}\bigr),
\end{align*}
where
\begin{align*}
    V(\theta) &= - E_{P_0}\bigl[ \nabla^2 \log p(X \mid \theta) \bigr], \\
    D(\theta) &= E_{P_0}\bigl[ \nabla \log p(X \mid \theta)\,\nabla \log p(X \mid \theta)^\top \bigr].
\end{align*}
When the model is correctly specified, the above ``sandwich'' covariance reduces to the inverse Fisher information matrix. Under misspecification, any unbiased estimator $\hat{\theta}$ of $\theta_0$ satisfies
\begin{equation*}
    \mathrm{Var}(\hat{\theta}) 
    \;\ge\; V(\theta_0)^{-1} D(\theta_0) V(\theta_0)^{-1},
\end{equation*}
so that the MLE is asymptotically efficient in the usual sandwich sense.

\subsection{Bernstein--von Mises (BvM) theorem for a Bayesian model}

We first introduce a \emph{local asymptotic normality (LAN)} condition for misspecified models, which will be useful in our later discussions. A model is said to satisfy a stochastic LAN condition around $\theta_0 \in \Theta$ relative to a rate $\delta_n \to 0$ if there exists a random vector $\Delta_{n,\theta_0}$, bounded in $P_0$-probability, such that for every compact set $K \subset \mathbb{R}^d$,
\begin{equation*}\label{eq-3}
\sup_{h \in K}
\left|
\log \prod_{i=1}^n \frac{p(X_i \mid \theta_0 + \delta_n h)}{p(X_i \mid \theta_0)}
- h^\top V(\theta_0)\Delta_{n,\theta_0}
- \frac{1}{2} h^\top V(\theta_0) h
\right|
\overset{P_0}{\longrightarrow} 0.
\end{equation*}

Under the above stochastic LAN condition, \citet{kleijn2012bernstein} prove an asymptotic normality result for the posterior distribution. Writing $\vartheta \sim \pi(\theta \mid X_{1:n})$ for a draw from the posterior, one has
\begin{align*}
    \sqrt{n}(\vartheta - \theta_0) - \Delta_{n,\theta_0}
    \stackrel{d}{\longrightarrow} N\bigl(0, V(\theta_0)^{-1}\bigr).
\end{align*}
For a misspecified model, the covariance of the limiting normal distribution differs from the sandwich covariance $V(\theta_0)^{-1} D(\theta_0) V(\theta_0)^{-1}$ introduced above. This discrepancy implies that, under misspecification, credible sets derived from the usual Bayesian posterior are in general not expected to have asymptotically valid frequentist coverage.


\subsection{Variational Bayes and asymptotic properties}
\label{sec-vb}

Variational Bayes approximates the posterior distribution by a member of a prespecified parametric family (the variational family) by minimizing the KL divergence between the posterior distribution and distributions in this family \citep{blei-reivew}. Variational Bayes has become one of the most popular approaches for posterior approximation due to its simplicity, generality, and computational efficiency. There is also an emerging literature providing theoretical guarantees for variational methods \citep{zhang2020convergence, pati2018statistical, alquier2020concentration, ohn2024adaptive}. 

We consider a setting in which the unknowns consist of latent variables
$Z_{1:n} = (Z_1,\ldots, Z_n)$
and a global parameter $\theta = (\theta_1,\ldots, \theta_d)$. 
Variational Bayes aims to approximate the joint posterior distribution $\pi(\theta, Z_{1:n} \mid X_{1:n})$ by solving the optimization problem
\begin{equation*}
    q^*(\theta, Z_{1:n})
    = \arg \min_{q \in \mathcal{Q}} \textup{KL}\bigl(q, \pi(\cdot \mid X_{1:n})\bigr),
\end{equation*}
where $\mathcal{Q}$ denotes the variational family. In this paper, we focus on the mean-field variational family
\begin{align}
\label{eq-MF-family}
    \mathcal{Q} 
    = \biggl\{
        q : 
        q(\theta, Z_{1:n})
        = \prod_{j=1}^d q_{\theta_j}(\theta_j) \prod_{i=1}^n q_{Z_i}(Z_i)
      \biggr\}.
\end{align}

Recently, frequentist asymptotic properties of variational Bayes approximations have been established, including consistency, contraction rates, and BVM
\citep{zhang2020convergence, pati2018statistical, yang2020alpha, alquier2020concentration, wang2019frequentist,wang2019variational}.  We briefly review the relevant results. 

\begin{defn}
The \emph{variational log-likelihood} is defined as
\begin{equation}
\label{eq-vbmle}
    \log p_{\textsc{vb}}(X \mid \theta)
    = \max_{q_Z}
      E_{q_Z}
      \bigl[
        \log p(X, Z \mid \theta)
        - \log q_Z(Z)
      \bigr].
\end{equation}
\end{defn}

\citet{wang2019variational} introduced a LAN-type condition for the variational log-likelihood. A model is said to satisfy a stochastic \emph{variational local asymptotic normality (VLAN)} condition around an interior point $\theta_0 \in \Theta$ relative to a rate $\delta_n \rightarrow 0$ if the following holds: there exists a random vector $\Delta_{n,\theta_0}$, bounded in $P_0$-probability, such that for every compact set $K \subset \mathbb{R}^d$,
\begin{equation*}
\sup_{h \in K}
\left|
\log \prod_{i=1}^n
\frac{p_{\textsc{vb}}(X_i \mid \theta_0 + \delta_n h)}
     {p_{\textsc{vb}}(X_i \mid \theta_0)}
- h^\top V_{\textsc{vb}}(\theta_0)\Delta_{n,\theta_0}
- \frac{1}{2} h^\top V_{\textsc{vb}}(\theta_0) h
\right|
\overset{P_0}{\longrightarrow} 0,
\end{equation*}
where
\begin{align}
\label{eq:vb_hessian}
    V_{\textsc{vb}}(\theta) 
    &= - E_{P_0}\bigl[\nabla^2 \log p_{\textsc{vb}}(X \mid \theta)\bigr],\\
\label{eq:vb_score_sq}
    D_{\textsc{vb}}(\theta)
    &= E_{P_0}\bigl[
        \nabla \log p_{\textsc{vb}}(X \mid \theta)\,
        \nabla \log p_{\textsc{vb}}(X \mid \theta)^\top
      \bigr].
\end{align}

Under the VLAN condition (together with additional regularity conditions), \citet{wang2019variational} show that the limiting distribution of $\theta$ under the variational posterior is Gaussian:
\begin{align}
\label{eq:bvm_vp}
    \sqrt{n}(\vartheta - \theta_0) - \Delta_{n,\theta_0}
    \stackrel{d}{\longrightarrow}
    N\bigl(0, (\tilde{V}_{\textsc{vb}}^0)^{-1}\bigr),
\end{align}
for $\vartheta \sim \int q^*(\theta, Z)\,dZ$, where $\tilde{V}_{\textsc{vb}}^0$ is the diagonal matrix that has the same diagonal entries as $V_{\textsc{vb}}(\theta_0)$. As is well-known and as shown in \eqref{eq:bvm_vp}, the VB covariance with a mean-field class is only diagonal.

\subsection{Asymptotic properties of BayesBag under model misspecification}

\citet{huggins2019robust} show that, under suitable regularity conditions on the log density $\log p(x \mid \theta)$, the BayesBag posterior $\tilde{\pi}(\theta \mid X_{1:n})$ in \eqref{eq-intro-bag} satisfies the following asymptotic normality:
for $\vartheta \sim \tilde{\pi}(\theta \mid X_{1:n})$,
\begin{align*}
    \sqrt{n}(\vartheta - \theta_0) - \Delta_n \,\big|\, X_{1:n}
    \overset{d}{\longrightarrow}
    N\!\left(
        0,\,
        \frac{1}{c} V(\theta_0)^{-1}
        + \frac{1}{c} V(\theta_0)^{-1} D(\theta_0) V(\theta_0)^{-1}
    \right),
\end{align*}
for some random vector $\Delta_n$ that is bounded in $P_0$-probability, where $c = \lim_{n \to \infty} M/n$ and $M$ is the bootstrap sample size.

In contrast to the original Bayesian posterior, whose limiting covariance under misspecification is $V(\theta_0)^{-1}$ (see BvM result in Section~\ref{sec-prelim}), the BayesBag posterior has asymptotic covariance
$
    \frac{1}{c} V(\theta_0)^{-1}
    + \frac{1}{c} V(\theta_0)^{-1} D(\theta_0) V(\theta_0)^{-1},
$
which is closer to the sandwich form $V(\theta_0)^{-1} D(\theta_0) V(\theta_0)^{-1}$ and therefore yields better-calibrated credible sets under model misspecification. In particular, taking $c = 1$ already leads to conservative uncertainty quantification. Moreover, as discussed in Section~\ref{sec:boot_size}, one can select $c$ in a principled way so that the resulting covariance approximates the sandwich form $V(\theta_0)^{-1} D(\theta_0) V(\theta_0)^{-1}$.

\section{Variational bagging: bagged variational posterior}
\label{sec-vbag}

In this section, we introduce our \emph{variational bagging} approach. For a bootstrap sample 
$X_{1:M}^* = (X_1^*, \ldots, X_M^*)$ generated from the original data 
$X_{1:n} = (X_1, \ldots, X_n)$, define
\begin{align*}
    q^*(\theta, Z_{1:M}^* \mid X_{1:M}^*)
    = \arg\min_{q \in \mathcal{Q}}
    \textup{KL}\bigl(q(\theta, Z_{1:M}^*),\, \pi(\theta, Z_{1:M}^* \mid X_{1:M}^*)\bigr)
\end{align*}
as the joint variational approximation to the posterior given this bootstrap dataset.
We focus on inference for $\theta$ via the marginal
\begin{align*}
    q^*(\theta \mid X_{1:M}^*)
    = \int q^*(\theta, Z_{1:M}^* \mid X_{1:M}^*) \, dZ_{1:M}^*.
\end{align*}
When $\mathcal{Q}$ is a mean-field family, $q^*(\theta \mid X_{1:M}^*)$ is easily obtained using the factorization structure.

We define the \emph{bagged variational posterior} as the average of the variational posteriors obtained from $B$ bootstrap samples:
\begin{equation*}
    q^{\textup{bvB}}(\theta \mid X_{1:n})
    = \frac{1}{B} \sum_{b=1}^B q^*(\theta \mid X_{(b)}^*),
\end{equation*}
where each $X_{(b)}^*$ denotes a bootstrap sample of size $M$. In our theoretical analysis of the bagged variational posterior, we focus on the mean-field variational family.

\subsection{Robust uncertainty quantification}

In this section, we derive the asymptotic distribution of the bagged variational posterior and discuss implications for uncertainty quantification. Recall the definitions of
$V_{\textsc{vb}}(\cdot)$ and $D_{\textsc{vb}}(\cdot)$ from \labelcref{eq:vb_hessian,eq:vb_score_sq}, that is,
\begin{align*}
    V_{\textsc{vb}}(\theta)
    &= -E_{P_0}\bigl[\nabla^2 \log p_{\textsc{vb}}(X \mid \theta)\bigr],\\
    D_{\textsc{vb}}(\theta)
    &= E_{P_0}\bigl[
        \nabla \log p_{\textsc{vb}}(X \mid \theta)\,
        \nabla \log p_{\textsc{vb}}(X \mid \theta)^\top
      \bigr].
\end{align*}

Notably, we show that under model misspecification, the asymptotic covariance of the bagged variational posterior contains an additional “sandwich” term
\begin{align*}
    (V_{\textsc{vb}}^0)^{-1} D_{\textsc{vb}}^0 (V_{\textsc{vb}}^0)^{-1},
    \quad
    V_{\textsc{vb}}^0 = V_{\textsc{vb}}(\theta_0),\;
    D_{\textsc{vb}}^0 = D_{\textsc{vb}}(\theta_0),
\end{align*}
on top of the covariance \((\tilde{V}_{\textsc{vb}}^0)^{-1}\) arising from the usual variational posterior, where \(\tilde{V}_{\textsc{vb}}^0\) is the diagonal matrix formed from the diagonal entries of \(V_{\textsc{vb}}^0\). In particular, \((\tilde{V}_{\textsc{vb}}^0)^{-1}\) corresponds to the covariance under the standard mean-field VB procedure, which only captures marginal variances.

\begin{thm}[Bernstein--von Mises theorem for the bagged variational posterior]
\label{thm-vbbag1}
Let \(\ell_\theta(x) = \log p_{\textsc{vb}}(x \mid \theta)\), and assume the following conditions hold:
\begin{enumerate}
    \item The map \(\theta \mapsto \ell_{\theta}(X_{1})\) is differentiable at \(\theta_0\) in probability.

    \item There exists an open neighborhood \(U\) of \(\theta_0\) and a function
    \(m_{\theta_0} : \mathcal{X} \to \mathbb{R}\) such that
    \(E_{P_0}\bigl(m_{\theta_0}^3\bigr) < \infty\), and for all
    \(\theta, \theta' \in U\),
    \[
        \bigl|\ell_\theta - \ell_{\theta'}\bigr|
        \le m_{\theta_0} \, \|\theta - \theta'\|_2
        \quad \text{a.s.\ }[P_0].
    \]

    \item As \(\theta \to \theta_0\),
    \[
        -E_{P_0}\bigl(\ell_\theta - \ell_{\theta_0}\bigr)
        = \frac{1}{2}(\theta - \theta_0)^\top V_{\textsc{vb}}^0 (\theta - \theta_0)
          + o\bigl(\|\theta - \theta_0\|_2^2\bigr).
    \]

    \item \(V_{\textsc{vb}}^0\) is invertible.

    \item For every \(\epsilon > 0\), there exists a sequence of tests \(\phi_n\) such that
    \begin{align*}
        &\int \phi_n(x_1,\dots,x_n)
          \prod_{i=1}^n p_0(x_i) \, dx_i \to 0, \\
        &\sup_{\theta:\,\|\theta - \theta_0\| > \epsilon}
          \int \bigl\{1 - \phi_n(x_1,\dots,x_n)\bigr\}
          \prod_{i=1}^n \frac{p_{\textsc{vb}}(x_i \mid \theta)}{p_{\textsc{vb}}(x_i \mid \theta_0)}\,
          p_0(x_i) \, dx_i \to 0.
    \end{align*}

    \item \(c = \lim_{n \to \infty} M/n \in (0,\infty)\), where \(M\) is the bootstrap sample size.
\end{enumerate}
Then, for \(\vartheta^\dag \sim q^{\textup{bvB}}(\theta \mid X_{1:n})\), we have
\begin{align}
\label{eq-5}
    \sqrt{n}(\vartheta^\dag - \theta_0) - \Delta_n
    \,\Big|\, X_{1:n}
    \overset{d}{\longrightarrow}
    N\!\left(
        0,\,
        \frac{1}{c}(\tilde{V}_{\textsc{vb}}^0)^{-1}
        + \frac{1}{c}(V_{\textsc{vb}}^0)^{-1} D_{\textsc{vb}}^0 (V_{\textsc{vb}}^0)^{-1}
    \right),
\end{align}
where
\[
    \Delta_n
    = n^{1/2} (V_{\textsc{vb}}^0)^{-1}
      (\mathbb{P}_n - P_0) \dot{\ell}_{\theta_0},
    \quad
    \mathbb{P}_n = n^{-1}\sum_{i=1}^n \delta_{X_i},
\]
and \(\tilde{V}_{\textsc{vb}}^0\) is the diagonal matrix with the same diagonal entries as \(V_{\textsc{vb}}^0\).
\end{thm}

\begin{rem}
The sandwich term
\(
  (V_{\textsc{vb}}^0)^{-1} D_{\textsc{vb}}^0 (V_{\textsc{vb}}^0)^{-1}
\)
in \cref{thm-vbbag1} can be viewed as the \emph{target covariance}: it corresponds to the covariance of the MLE  defined via the variational log-likelihood \eqref{eq-vbmle}. This term is typically \emph{non-diagonal}, even when a mean-field variational family is used. It is precisely this additional sandwich term that drives robustness and more accurate uncertainty quantification of VB bagging: even under model misspecification, variational bagging attempts to recover the “best” covariance structure allowed by the model and the variational approximation.

When the model is specified correctly, we have
\(
  \textup{diag}\bigl((V_{\textsc{vb}}^0)^{-1} D_{\textsc{vb}}^0 (V_{\textsc{vb}}^0)^{-1}\bigr)
  = \textup{diag}\bigl((\tilde{V}_{\textsc{vb}}^0)^{-1}\bigr). 
\) That is, these two  terms share the same diagonal covariance. 
The covariance in \eqref{eq-5} can be decomposed as
\[
    \frac{1}{c}
    \Bigl(
        (V_{\textsc{vb}}^0)^{-1} D_{\textsc{vb}}^0 (V_{\textsc{vb}}^0)^{-1}
        - (\tilde{V}_{\textsc{vb}}^0)^{-1}
    \Bigr)
    + \frac{2}{c}(\tilde{V}_{\textsc{vb}}^0)^{-1}.
\]
The first term above contains only off-diagonal entries, while the second term is purely diagonal. In this case, one can recover the off-diagonal entries of the target covariance 
    \((V_{\textsc{vb}}^0)^{-1} D_{\textsc{vb}}^0 (V_{\textsc{vb}}^0)^{-1}\)
    from the bagged VB posterior (e.g., with \(c = 1\)), and for the diagonal term we need to make the simple adjustment by multiplying them by 1/2.
Combining these two pieces yields an estimator of the full target covariance.

When the model is misspecified, we can still recover the off-diagonal structure by setting \(c = 1\) and extracting only the off-diagonal entries from the bagged VB covariance. For the diagonal entries, we can use the choice of \(M\) described in Section~\ref{sec:boot_size} to match the desired marginal variances.
\end{rem}

\begin{rem}
If we instead use a variational family of Gaussian distributions with general covariance matrices,
\(\mathcal{Q} = \{q : q(\theta) = N(\mu, \Sigma)\}\), i.e., without a mean-field restriction, then by inspecting the proof of \cref{thm-vbbag1} one can see that the limiting distribution of the bagged variational posterior becomes
\begin{align*}
    \sqrt{n}(\vartheta^\dag - \theta_0) - \Delta_n
    \,\Big|\, X_{1:n}
    \overset{d}{\longrightarrow}
    N\!\left(
        0,\,
        \frac{1}{c} (V_{\textsc{vb}}^0)^{-1}
        + \frac{1}{c} (V_{\textsc{vb}}^0)^{-1} D_{\textsc{vb}}^0 (V_{\textsc{vb}}^0)^{-1}
    \right).
\end{align*}
\end{rem}

The next corollary is a special case of \cref{thm-vbbag1} when variational inference is not conducted over a latent variable \(Z\) (that is, we marginalize \(Z\) in the model), so that \(p_{\textsc{vb}}(x \mid \theta) = p(x \mid \theta)\).

\begin{cor}[BvM theorem without latent variables]
\label{thm:bvm_no_latent}
Assume that \(p_{\textsc{vb}}(x \mid \theta) = p(x \mid \theta)\). Let \(V^0 = V(\theta_0)\) and \(D^0 = D(\theta_0)\), and let \(\tilde{V}^0\) be the diagonal matrix with the same diagonal entries as \(V^0\). Then, under the same conditions as in \cref{thm-vbbag1}, for \(\vartheta^\dag \sim q^{\textup{bvB}}(\theta \mid X_{1:n})\),
\begin{align*}
    \sqrt{n}(\vartheta^\dag - \theta_0) - \Delta_n
    \,\Big|\, X_{1:n}
    \overset{d}{\longrightarrow}
    N\!\left(
        0,\,
        \frac{1}{c}(\tilde{V}^0)^{-1}
        + \frac{1}{c}(V^0)^{-1} D^0 (V^0)^{-1}
    \right),
\end{align*}
where
\[
    \Delta_n
    = n^{1/2}(V^0)^{-1}(\mathbb{P}_n - P_0)\dot{\ell}_{\theta_0}.
\]
\end{cor}

In \cref{thm:bvm_no_latent}, the sandwich term \((V^0)^{-1} D^0 (V^0)^{-1}\) does not depend on the choice of variational family and coincides with the usual sandwich covariance for the MLE. Moreover, the full asymptotic covariance of the bagged variational posterior is larger than that of the MLE, with the difference given by \((\tilde{V}^0)^{-1}/c\). Thus, the bagged variational posterior yields conservative uncertainty quantification: asymptotically, it is never overconfident. This is rigorously formalized in the following corollary.

\begin{cor}[No overconfident credible sets]
\label{thm:confidence}
Assume the same conditions as in \cref{thm:bvm_no_latent} and take \(M = n\) (so that \(c = 1\)). Consider the ellipsoid
\begin{align*}
    C(r)
    = \left\{
        \theta \in \Theta :
        n(\theta - \hat{\theta}_{\textup{mle}})^\top \widehat{\Sigma}^{-1}
        (\theta - \hat{\theta}_{\textup{mle}})
        \le r^2
      \right\}
\end{align*}
with radius \(r > 0\), where \(\hat{\theta}_{\textup{mle}}\) is the MLE of \(\theta_0\) and
\(\widehat{\Sigma}\) is a consistent estimator of the asymptotic covariance
\[
    \Sigma_0
    = (\tilde{V}^0)^{-1} + (V^0)^{-1} D^0 (V^0)^{-1}
\]
of the bagged variational posterior. For \(\alpha \in (0,1)\), let \(r_{n,1-\alpha}\) be such that
\(C(r_{n,1-\alpha})\) is a \((1-\alpha)\)-credible set for \(\theta_0\) under the bagged variational posterior, i.e.,
\[
    Q^{\textup{bvB}}\bigl(\vartheta^\dag \in C(r_{n,1-\alpha})\bigr) = 1 - \alpha.
\]
Then
\begin{align*}
    \lim_{n \to \infty}
    P_0\bigl(\theta_0 \in C(r_{n,1-\alpha})\bigr)
    \ge 1 - \alpha.
\end{align*}
\end{cor}

\subsection{Valid variational Bayes uncertainty quantification}
\label{sec-vbbag}

Mean-field variational Bayes is well known for providing a fast and tractable approximation of the Bayesian posterior. As shown by \citet{wang2019frequentist} (see Equation~\eqref{eq:bvm_vp}), mean-field approximations are asymptotically normal under standard conditions, but they \emph{only approximate the diagonal terms of the covariance structure}, ignoring dependence among parameters. For this reason, although variational Bayes often delivers accurate first-order inference (e.g., point estimates), it is generally unsuitable for second-order inference such as uncertainty quantification, even when the model is correctly specified.

An interesting consequence of Theorem~\ref{thm-vbbag1} is that the variational bagging procedure can address this limitation by recovering the off-diagonal elements of the covariance structure that are wiped out in the standard mean-field variational approximation. In particular, when the model is correctly specified, variational bagging yields asymptotically valid uncertainty quantification, comparable to that of the standard Bayesian posterior.  We briefly discussed this in Remark 1 for the case of a well-specified model with the presence of latent variables, and the following deals with the case when there is no latent variable. 

\begin{cor}[BvM theorem when the model is correctly specified]
\label{cor-vbbag1}
Let $\ell_\theta(x) = \log p_{\textsc{vb}}(x \mid \theta)$ and assume Assumptions 1--5 in Theorem~\ref{thm-vbbag1} hold, with $M = n$. In this case $\theta_0$ is the true parameter so that $P_0 = P_{\theta_0}$. Then, for $\vartheta^\dag \sim q^{\textup{bvB}}(\theta \mid X_{1:n})$, we have
\begin{align*}
    \sqrt{n}(\vartheta^\dag - \theta_0) - \Delta_n
    \,\big|\, X_{1:n}
    \overset{d}{\longrightarrow}
    N\bigl(0,\,
        (\tilde{V}_{\textsc{vb}}^0)^{-1}
        + (V_{\textsc{vb}}^0)^{-1}
    \bigr),
\end{align*}
where $V_{\textsc{vb}}^0 = V_{\textsc{vb}}(\theta_0)$ and $\tilde{V}_{\textsc{vb}}^0$ is the diagonal matrix with the same diagonal entries as $V_{\textsc{vb}}^0$.
\end{cor}

When the model is correctly specified and the usual regularity conditions hold, the Fisher information structure implies that
\[
    \text{diag}\bigl((V_{\textsc{vb}}^0)^{-1}\bigr)
    =
    \text{diag}\bigl((\tilde{V}_{\textsc{vb}}^0)^{-1}\bigr),
\]
so that the limiting covariance in Corollary~\ref{cor-vbbag1} can be decomposed as
\begin{align*}
    (\tilde{V}_{\textsc{vb}}^0)^{-1}
    + (V_{\textsc{vb}}^0)^{-1}
    &= 2\,\text{diag}\bigl((V_{\textsc{vb}}^0)^{-1}\bigr)
       \;+\; \text{offdiag}\bigl((V_{\textsc{vb}}^0)^{-1}\bigr),
\end{align*}
where $\text{offdiag}(A)$ denotes the matrix obtained from $A$ by zeroing out its diagonal. The first term above only has nonzero diagonal entries, while the second term only contains off-diagonal entries. Thus:
\begin{itemize}
    \item the \emph{off-diagonal} part of the target covariance $(V_{\textsc{vb}}^0)^{-1}$ is recovered by the bagged variational posterior, and
    \item the \emph{diagonal} part is inflated by a factor of~2 relative to the target covariance.
\end{itemize}
To recover the same covariance as in the standard Bayesian BvM result, one can simply rescale the diagonal entries of the bagged VB covariance by $1/2$, keeping the off-diagonal entries unchanged. In this way, variational bagging can serve as a general tool to improve and calibrate uncertainty quantification based on mean-field variational Bayes, restoring both the correct marginal variances (after a simple adjustment) and the dependence structure.

\subsubsection{Toy example illustration}\label{sec-toy}

We illustrate the implication of Corollary~\ref{cor-vbbag1} with a simple toy example. Consider estimation of the mean vector 
$\mu$ for two-dimensional Gaussian data $X_1,\dots,X_{500}$, where
\[
    X_i \sim N(\mu, \Sigma), \quad
    \mu = (-1, 1)^\top,\quad
    \Sigma = 
    \begin{pmatrix}
        1 & 0.5 \\
        0.5 & 1
    \end{pmatrix}.
\]
We place a conjugate prior on $\mu$ and compare three posterior approximations: (i) Hamiltonian Monte Carlo (HMC), (ii) a mean-field variational approximation, and (iii) our variational bagging approach.

Figure~\ref{fig:2d} shows the resulting $95\%$ posterior credible regions. The mean-field variational posterior provides a reasonable approximation of the posterior mean of $\mu$, but it fails to capture the correlation structure between the components of $\mu$ and yields an elliptical region aligned with the coordinate axes. In contrast, the bagged variational posterior closely matches the full Bayesian posterior obtained by HMC, successfully recovering the correlation and orientation of the credible region. This illustrates how variational bagging can substantially improve uncertainty quantification over standard mean-field variational Bayes while retaining its computational benefits.

\begin{figure}[htbp!]
    \centering
    \includegraphics[width=0.5\textwidth]{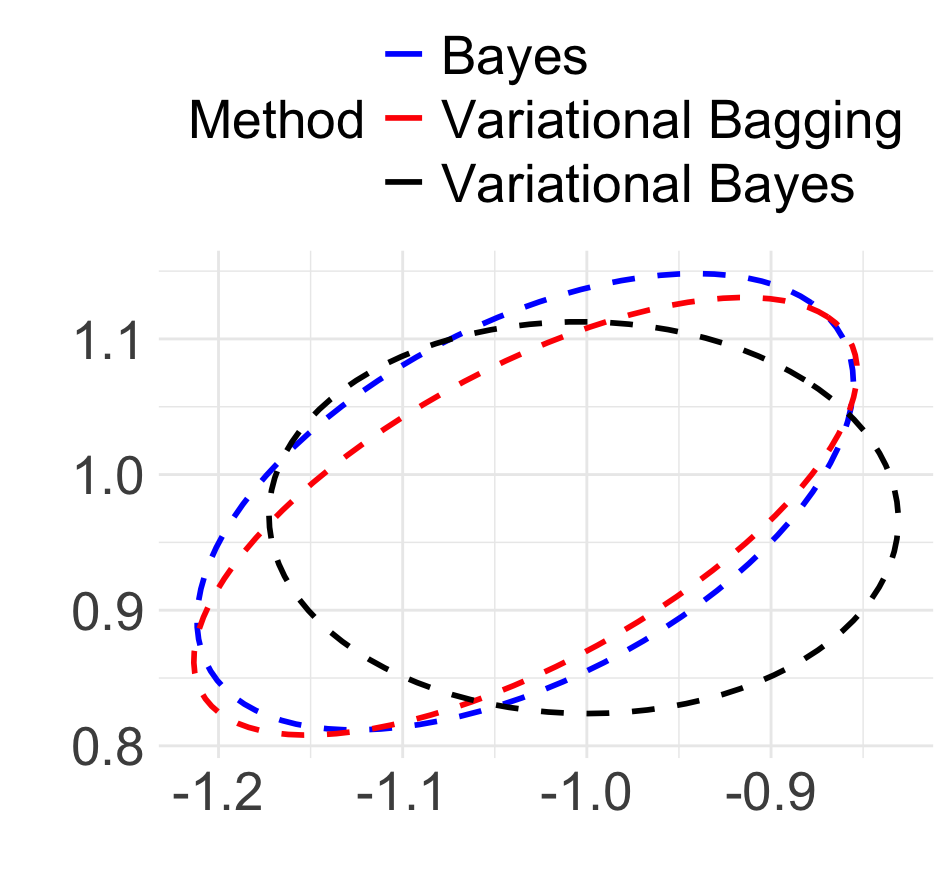}
    \caption{$95\%$ posterior credible regions for a two-dimensional Gaussian mean: comparison of HMC, mean-field VB, and variational bagging. }
    \label{fig:2d}
\end{figure}

\subsection{Contraction rates of the bagged variational posterior}

Our BvM theorem (Theorem~\ref{thm-vbbag1}) implies that the bagged variational posterior contracts to $\theta_0$ at the parametric rate $n^{-1/2}$ for finite-dimensional parametric models satisfying its assumptions. In this section, we extend this result to a general setup that encompasses nonparametric models, and we consider convergence in Hellinger distance. Let
\[
    H(P_1, P_2)
    =
    \left(
        \int \bigl(\sqrt{dP_1/dP_2} - 1\bigr)^2 \, dP_2
    \right)^{1/2}
\]
denote the Hellinger distance between two probability measures $P_1$ and $P_2$. To quantify the corresponding convergence rate, we use the notion of bracketing Hellinger metric entropy.

We say that a finite collection of pairs of functions
\(\{(f_i^L, f_i^U) : i = 1,\dots, N\}\) is a $\delta$-bracketing of a function space
\(\mathcal{F}\) if
\[
    \bigl\|(f_i^U)^{1/2} - (f_i^L)^{1/2}\bigr\|_2 \le \delta,
    \quad i = 1,\dots, N,
\]
and for any \(f \in \mathcal{F}\) there exists \(i \in \{1,\dots, N\}\) such that
\(f_i^L \le f \le f_i^U\). The $\delta$-bracketing Hellinger metric entropy of
\(\mathcal{F}\), denoted by $\mathcal{H}_B(\delta, \mathcal{F})$, is defined as the logarithm of the cardinality of a minimal $\delta$-bracketing.

\begin{assumption}
\label{assume:contraction}
Let $(\epsilon_n)_{n=1,2,\dots,}$ be a positive sequence such that $\epsilon_n\to0$ and $n \epsilon_n^2 \to \infty$ as $n \to \infty$. Assume the following:
\begin{enumerate}
    \item[\textup{(1)}] \textbf{(Prior mass)} There exists a constant $C_1 > 0$ such that
    \begin{align*}
        \Pi\Bigl(
            \theta \in \Theta :
            \textup{KL}(P_0, P_\theta) \le \epsilon_n^2,\;
            \textup{KLV}(P_0, P_\theta) \le \epsilon_n^2
        \Bigr)
        \;\ge\; \exp(-C_1 n \epsilon_n^2),
    \end{align*}
    where
    \(\textup{KLV}(P_0, P_\theta)
        = \int \bigl(\log(dP_\theta/dP_0)\bigr)^2 \, dP_0\).

    \item[\textup{(2)}] \textbf{(Sieve and complexity)} There exists a subset $\Theta_n \subset \Theta$ such that
    \begin{align*}
        \Pi(\Theta \setminus \Theta_n)
        \le \exp\bigl(-(C_1 + 4)n \epsilon_n^2\bigr)
    \end{align*}
    and, for some constants $C_2 > 0$ and $C_3 > 0$,
    \begin{align*}
        \int_{\epsilon^2/2^8}^{\sqrt{2}\,\epsilon}
            \mathcal{H}_B^{1/2}
            \bigl(u/C_2,\,
                  \{p(\cdot \mid \theta) : \theta \in \Theta_n\}
            \bigr)
        \, du
        \;\le\;
        C_3 \sqrt{n}\,\epsilon^2
    \end{align*}
    for any $\epsilon > \epsilon_n$.

    \item[\textup{(3)}] \textbf{(Variational family)} There exists a constant $C_4 > 0$ such that
    \begin{align*}
        \inf_{q \in \mathcal{Q}}
        \left[
            n \int \textup{KL}(P_{\theta}, P_{0})\, q(\theta)\,d\theta
            + \textup{KL}(q, \pi)
        \right]
        \le C_4 n \epsilon_n^2.
    \end{align*}

    \item[\textup{(4)}] \textbf{(Bootstrap sample size)} $M = n$.
\end{enumerate}
\end{assumption}

The “prior mass and testing’’ method, first developed in the seminal paper \citet{ghosal2000convergence}, is a powerful and general tool for deriving contraction rates of posterior distributions. We retain a prior mass condition in Assumption~\ref{assume:contraction}(1), which plays the same role as in the prior mass and testing framework. However, in our setting it is not straightforward to construct suitable tests based on bootstrapped samples. To the best of our knowledge, there is no principled general approach for test construction for Bayesian procedures involving bootstrap-based posteriors.

Instead, we follow the strategy of \citet{shen2001rates} with a suitable modification to handle bootstrap-weighted likelihoods. \citet{han2019statistical} adopt a similar idea, but their result is restricted to parametric models. In our nonparametric setting, we use the complexity condition on the Hellinger metric entropy in Assumption~\ref{assume:contraction}(2), adapted from \citet{shen2001rates}, to control the behavior of the empirical process of the “bootstrap-weighted’’ log-likelihood ratio. The third assumption on the variational family, which has been employed in several related works such as \citet{ohn2024adaptive} and \citet{zhang2020convergence}, controls the variational approximation error between the variational posterior and the original posterior distribution. We show that this assumption, together with the bootstrap sample size condition in Assumption~\ref{assume:contraction}(4), can still be used to bound the variational approximation error even when a bootstrapped sample is used.

\begin{thm}[Contraction rate]
\label{thm:contraction}
Suppose that Assumption~\ref{assume:contraction} holds. Then the bagged variational posterior satisfies
\begin{align*}
    E\Bigl[
        Q^{\textup{bvB}}
        \bigl(
            H^2(P_{\theta}, P_{0})
            \ge M_n \epsilon_n^2 \log n
        \bigr)
    \Bigr]
    \;\longrightarrow\; 0,
\end{align*}
as $n \to \infty$ for any diverging sequence $(M_n)$ with $M_n \to \infty$, where the expectation is taken with respect to $P_0^{\otimes n}$.
\end{thm}

\subsection{Illustrating examples}
\label{sec-appmodel}

In this section, we demonstrate our theory by considering two simple examples: a multivariate Gaussian example and a two-component finite mixture model.

\subsubsection{Multivariate Gaussian mean}

Consider again the toy example in Section~\ref{sec-toy}. That is,
$X_{1:n} \in \mathbb{R}^2$ with
\[
    X_i \sim N(\mu_0, \Lambda_0^{-1}),
\]
where $\Lambda_0$ is known and we are interested in estimating the mean vector $\mu_0$. For the posterior distribution $\pi(\mu \mid X_{1:n})$, the Bernstein--von Mises theorem yields
\begin{align*}
    \sqrt{n}(\mu - \bar{X}_n)
    \stackrel{d}{\longrightarrow}
    N(0, \Lambda_0^{-1}),
\end{align*}
for $\mu \sim \pi(\mu \mid X_{1:n})$, where $\bar{X}_n$ is the sample mean.

For variational Bayes, we consider the mean-field variational family, which assumes
$q(\mu) = q_1(\mu_1)\,q_2(\mu_2)$. By the BvM theorem for the variational posterior in \citet{wang2019variational}, we have
\begin{align*}
    \sqrt{n}(\mu - \bar{X}_n)
    \stackrel{d}{\longrightarrow}
    N\Biggl(
        0,\,
        \begin{pmatrix}
            \Lambda_{022}^{-1} & 0 \\
            0 & \Lambda_{011}^{-1}
        \end{pmatrix}/\text{det}(\Lambda_0)
    \Biggr),
\end{align*}
for $\mu$ drawn from the mean-field variational posterior $q^*(\mu)$, where we decompose the true precision matrix as
\[
    \Lambda_0 =
    \begin{pmatrix}
        \Lambda_{011} & \Lambda_{012} \\
        \Lambda_{021} & \Lambda_{022}
    \end{pmatrix}.
\]
Compared with the asymptotic distribution of the exact posterior, we see that variational Bayes ignores the correlation between $\mu_1$ and $\mu_2$.

Now, applying Corollary~\ref{cor-vbbag1}, we obtain
\begin{align*}
    \sqrt{n}(\mu - \bar{X}_n)
    \,\big|\, X_{1:n}
    \stackrel{d}{\longrightarrow}
    N\Biggl(
        0,\,
        \begin{pmatrix}
            \Lambda_{022}^{-1} & 0 \\
            0 & \Lambda_{011}^{-1}
        \end{pmatrix}/\text{det}(\Lambda_0)
        + \Lambda_0^{-1}
    \Biggr),
\end{align*}
for $\mu$ drawn from the bagged variational posterior $q^{\textup{bvB}}(\mu \mid X_{1:n})$.

Therefore, using variational bagging, we recover the off-diagonal (correlation) structure of the true posterior covariance matrix via the $\Lambda_0^{-1}$ term, while the diagonal terms are inflated relative to the target covariance. As discussed in Section~\ref{sec-vbbag}, a simple correction of rescaling the diagonal entries by a factor of $1/2$ recovers the full target covariance, providing a concrete illustration of how variational bagging improves uncertainty quantification over standard mean-field variational Bayes in this simple Gaussian setting.

\subsubsection{Bayesian mixture models}

In this example, we consider model misspecification in a finite mixture model. For technical simplicity, we assume that our working inference model is a symmetric two-component Gaussian mixture with unit variance, 
\begin{align*}
    p(x \mid \theta)
    = \frac{1}{2} N(x; \theta, 1)
      + \frac{1}{2} N(x; -\theta, 1),
\end{align*}
with $\theta > 0$, while the true data-generating distribution $P_0$ is not necessarily in this model class.

The above mixture model can be equivalently written in hierarchical form as
\begin{align*}
    X_i \mid Z_i &\sim N\bigl((2Z_i - 1)\theta, 1\bigr),\\
    Z_i &\sim \mathrm{Bernoulli}(1/2),
\end{align*}
so we can conduct variational inference jointly over the parameter $\theta$ and latent variables $Z_1,\dots,Z_n$. We consider a mean-field variational family in which each $Z_i$ has a degenerate (point-mass) distribution at either $0$ or $1$, which is analogous to a hard clustering procedure.

In this case, it is straightforward to see that the variational log-likelihood is given by
\begin{align*}
    \ell(\theta)
    = \log p_{\textsc{vb}}(X \mid \theta)
    &= \sup_{q(Z)}
       \int q(Z)\,\log \frac{p(X, Z \mid \theta)}{q(Z)}\, dZ \\
    &= \max_{z \in \{0,1\}} \log p(X, z \mid \theta) \\
    &= \frac{1}{2}\max\bigl\{-(X-\theta)^2,\; -(X+\theta)^2\bigr\}
       - \frac{1}{2}\log(2\pi) \\
    &= -\frac{1}{2}\bigl(X - \operatorname{sign}(X)\theta\bigr)^2
       - \frac{1}{2}\log(2\pi).
\end{align*}
Differentiating twice in $\theta$ gives
\begin{align*}
    V_{\textsc{vb}}(\theta)
    = -E_{P_0}\Bigl[\frac{d^2}{d\theta^2}\,\ell(\theta)\Bigr]
    = 1.
\end{align*}

It is then easy to verify that the first through fourth conditions of \cref{thm-vbbag1} hold for this model. The fifth condition (existence of suitable tests) follows from Theorem~1 of \citet{westling2019beyond}, which establishes consistency of the maximum variational likelihood estimator in this setting. In our case, the maximum ``variational'' likelihood estimator is
\begin{align*}
    \hat{\theta}_n
    = \frac{1}{n} \sum_{i=1}^n X_i\,\operatorname{sign}(X_i),
\end{align*}
and this estimator is consistent for the pseudo-true parameter $\theta_0$ under model misspecification.

Applying \cref{thm-vbbag1}, we obtain, for
$\vartheta^\dag \sim q^{\textup{bvB}}(\theta \mid X_{1:n})$,
\begin{align*}
    \sqrt{n}(\vartheta^\dag - \hat{\theta}_n)
    \,\big|\, X_{1:n}
    \overset{d}{\longrightarrow}
    N\!\left(
        0,\,
        \frac{1}{c}\left\{
            1 + E_{P_0}\bigl[(X - \operatorname{sign}(X)\theta_0)^2\bigr]
        \right\}
    \right),
\end{align*}
where $c = \lim_{n \to \infty} M/n$ is the limiting ratio of bootstrap sample size to the original sample size. This illustrates how, even under model misspecification in a mixture setting, the bagged variational posterior enjoys a well-defined asymptotic distribution with a variance that incorporates a sandwich-type correction term.

\section{Simulation study}
\label{sec-sim}

In this section, we conduct a simulation study in which we apply variational bagging to several examples, including the multivariate Gaussian model, a finite Gaussian mixture model, sparse linear regression, regression based on feedforward deep neural networks, and a bagged VAE (variational autoencoder). We first discuss some practical aspects of variational bagging, such as the choice of bootstrap sample size and the number of bootstrap replicates.

 \subsection{Bootstrap sample size and number of bootstrap samples}
\label{sec:boot_size}

If we have strong confidence in our model specification, we may set $M = n$ and use mean-field VB to learn the off-diagonal terms of the covariance and take 1/2 of the diagonal terms. Under model misspecification, however, as illustrated in the previous examples, choosing $M = n$ (i.e., $c = 1$) may not yield the desired robust, sandwich-type covariance.

Let $\tilde{v}_n$ and $\tilde{v}_n^*$ denote, respectively, the standard and bagged variational posterior variances of a real-valued functional $f(\theta)$, where the bagged variational posterior is computed with $M = n$. In the asymptotic BvM setting, the bagged variational covariance (with a general $c$) behaves like
\[
    \frac{1}{c} (V_{\textsc{vb}}^0)^{-1}
    + \frac{1}{c} (V_{\textsc{vb}}^0)^{-1} D_{\textsc{vb}}^0 (V_{\textsc{vb}}^0)^{-1},
\]
while the “target’’ (sandwich) covariance is
\[
    (V_{\textsc{vb}}^0)^{-1} D_{\textsc{vb}}^0 (V_{\textsc{vb}}^0)^{-1}.
\]
At the level of a scalar functional $f(\theta)$, this corresponds to finding $c$ such that
\[
    \frac{1}{c} \tilde{v}_n + \frac{1}{c} \bigl(\tilde{v}_n^* - \tilde{v}_n\bigr)
    \approx \tilde{v}_n^* - \tilde{v}_n,
\]
which yields
\[
    \frac{1}{c} \tilde{v}_n^* = \tilde{v}_n^* - \tilde{v}_n
    \quad \Longrightarrow \quad
    c = \frac{\tilde{v}_n^*}{\tilde{v}_n^* - \tilde{v}_n}.
\]
Hence the corresponding optimal bootstrap sample size is
\[
    M^* = c \cdot n
    = \frac{\tilde{v}_n^*}{\tilde{v}_n^* - \tilde{v}_n}\, n.
\]
We therefore suggest the plug-in estimator
\begin{equation}
\label{eq-4}
    \hat{M}^*
    = \frac{\tilde{v}_n^*}{\tilde{v}_n^* - \tilde{v}_n}\, n.
\end{equation}

In finite-sample settings, we also need to account for prior influence when choosing a suitable bootstrap sample size. Following \citet{huggins2019robust}, we define a finite-sample version of the optimal bootstrap size, denoted $\hat{M}^*_{\mathrm{fs},\mathrm{opt}}$.

Let $v_0$ denote the prior variance of $f(\theta)$ and define
\[
    \hat{\sigma}_{\circ}^2
    := n\, v_0 \tilde{v}_n \big/ \bigl(v_0 - \tilde{v}_n\bigr),
\]
and
\begin{equation}
    \hat{s}_{\circ}^2
    := \frac{v_0^2}{\bigl(v_0 - \tilde{v}_n\bigr)^2}
       \bigl(\tilde{v}_n^* - \tilde{v}_n\bigr)\, n.
\end{equation}
Then the finite-sample optimal bootstrap size is estimated by
\begin{equation}
\label{eq:fs_opt}
    \hat{M}_{\mathrm{fs},\mathrm{opt}}
    := \frac{n}{2}
       + \frac{n \hat{\sigma}_{\circ}^2}{2 \hat{s}_{\circ}^2}
       - \frac{\hat{\sigma}_{\circ}^2}{v_0}
       + \left\{
           \left(\frac{n}{2}
                 + \frac{n \hat{\sigma}_{\circ}^2}{2 \hat{s}_{\circ}^2}
           \right)^2
           - \frac{n \hat{\sigma}_{\circ}^2}{v_0}
         \right\}^{1/2}.
\end{equation}
For the derivation of \eqref{eq:fs_opt}, we refer to Appendix~E of \citet{huggins2019robust}.

Regarding the number of bootstrap samples $B$, \citet{huggins2019robust} recommend $B \approx 50$ to $100$ for BayesBag due to the computational cost of repeated MCMC. In variational bagging, by contrast, computing each variational posterior is typically much faster than running MCMC, allowing substantially larger $B$ (e.g., $B \approx 200$). In our experiments, however, we find that even a relatively small number of bootstrap samples, such as $B = 5$, is often sufficient to obtain robust and well-calibrated uncertainty quantification.

\subsection{Uncertainty quantification for a multivariate Gaussian}

Using the same toy example as in Section~\ref{sec-toy}, we generate two-dimensional Gaussian data $X_1,\dots,X_n$, where
\[
    X_i \sim N(\mu, \Sigma), \quad
    \mu = (-1, 1)^\top,\quad
    \Sigma =
    \begin{pmatrix}
        1 & 0.5 \\
        0.5 & 1
    \end{pmatrix}.
\]
We vary the number of bootstrap replicates $B$ and the sample size $n$ to assess the effectiveness of uncertainty quantification under variational bagging. Specifically, we take
\[
    B \in \{5, 10, 20, 30, 50\}, \quad
    n \in \{50, 100, 200, 300, 500, 1000\}.
\]
Because Corollary~\ref{cor-vbbag1} is an asymptotic result, our goal here is to identify practically reasonable choices of $B$ and $n$ for which the asymptotic behavior is already well approximated.

We place a conjugate prior on $\mu$ and approximate the posterior using three methods: (i) Hamiltonian Monte Carlo (HMC) with 2 chains and 2000 posterior draws, (ii) mean-field variational Bayes (MFVB), and (iii) variational bagging based on MFVB. For variational bagging, we compute MFVB for each bootstrap dataset and average the corresponding variational posteriors; runs in which the MFVB algorithm fails to converge are discarded.

\begin{figure}[htbp!]
    \centering
    \includegraphics[width=1\linewidth]{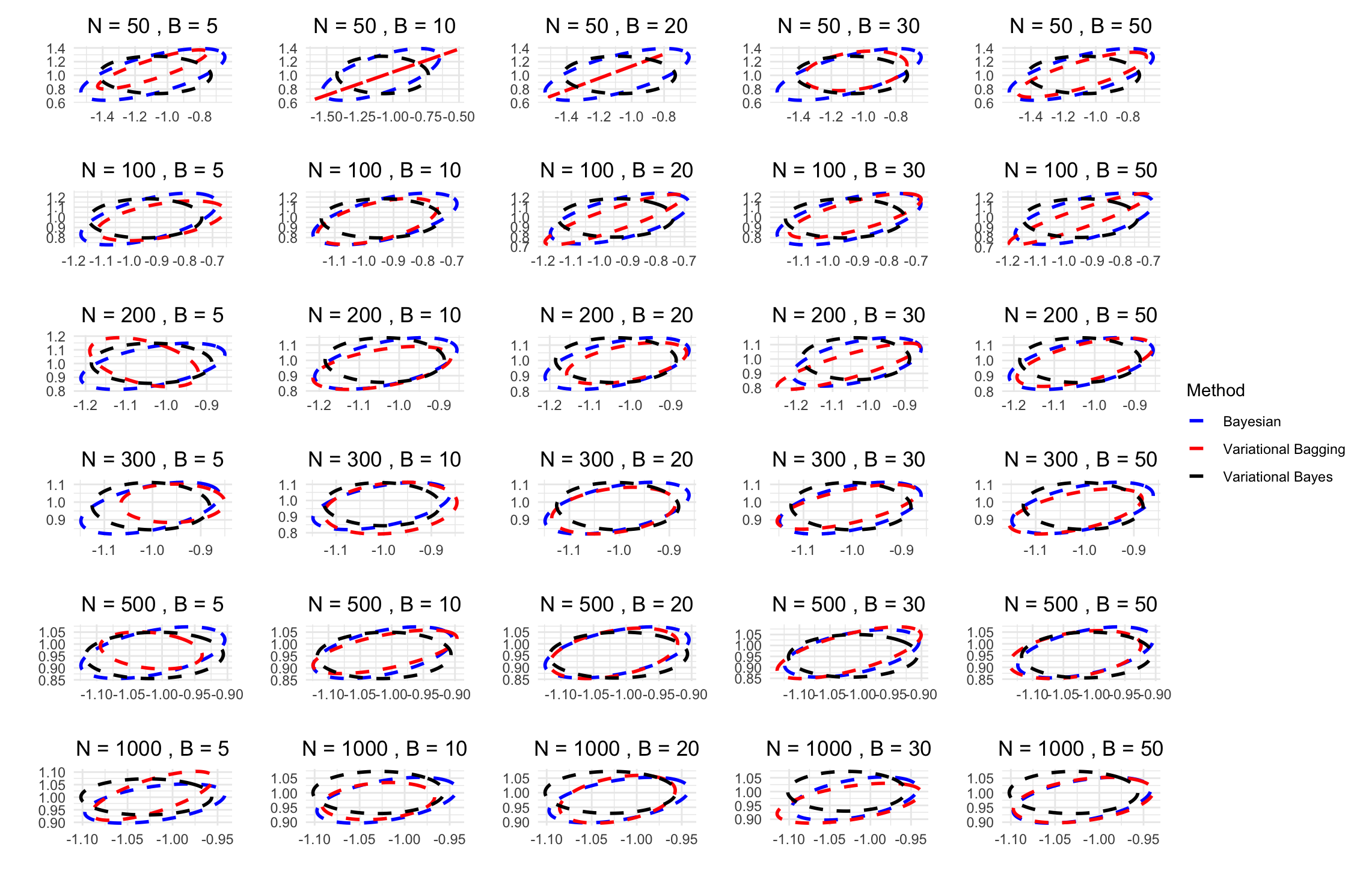}
    \caption{$95\%$ posterior credible regions for the Gaussian mean under HMC, MFVB, and variational bagging, for varying sample size $n$ and number of bootstrap replicates $B$.}
    \label{fig:2d-gauss-cmp}
\end{figure}

From Figure~\ref{fig:2d-gauss-cmp}, we observe that when the number of bootstrap replicates $B$ is at least $30$, the credible regions obtained from variational bagging closely match those from the full Bayesian posterior, even for relatively small sample sizes. In particular, the results suggest that our theoretical guarantees are already informative for sample sizes as small as $n = 50$, and that a moderate number of bootstrap replicates (around $B \ge 30$) suffices to capture the improved uncertainty quantification predicted by our theory.

\subsection{Gaussian mixture model}

This simulation study considers the Gaussian mixture model
\begin{align*}
    X_i \sim \sum_{k=1}^K \pi_k N(\mu_k, \sigma_k^2), \quad i = 1,\ldots, n.
\end{align*}
For the prior, we use
\[
    \pi \sim \text{Dirichlet}(\alpha), \quad
    \mu_k \mid \sigma_k^2 \sim N(0, \nu_0 \sigma_k^2), \quad
    \sigma_k^2 \sim IG(a,b).
\]

We generate data from heavy-tailed mixture distributions such as $t$ mixtures and double-exponential mixtures, but fit a Gaussian mixture model to these data. We also fit the same data under the correctly specified (true) model to evaluate the performance of variational bagging.

For the full Bayesian posterior, we use \emph{Stan} (R interface), based on Hamiltonian Monte Carlo (HMC), with 2000 iterations and 1000 burn-in per chain. For variational Bayes, we implement a CAVI algorithm. For bagging, we choose $B = 50$ bootstrap replicates and set the bootstrap sample size to be $\hat{M}^*$ from \eqref{eq-4}. To accelerate computation, variational fits for different bootstrap samples are run in parallel.

\begin{figure}[htbp!]
    \centering
    \begin{subfigure}[b]{0.3\textwidth}
        \centering
        \includegraphics[width=\textwidth]{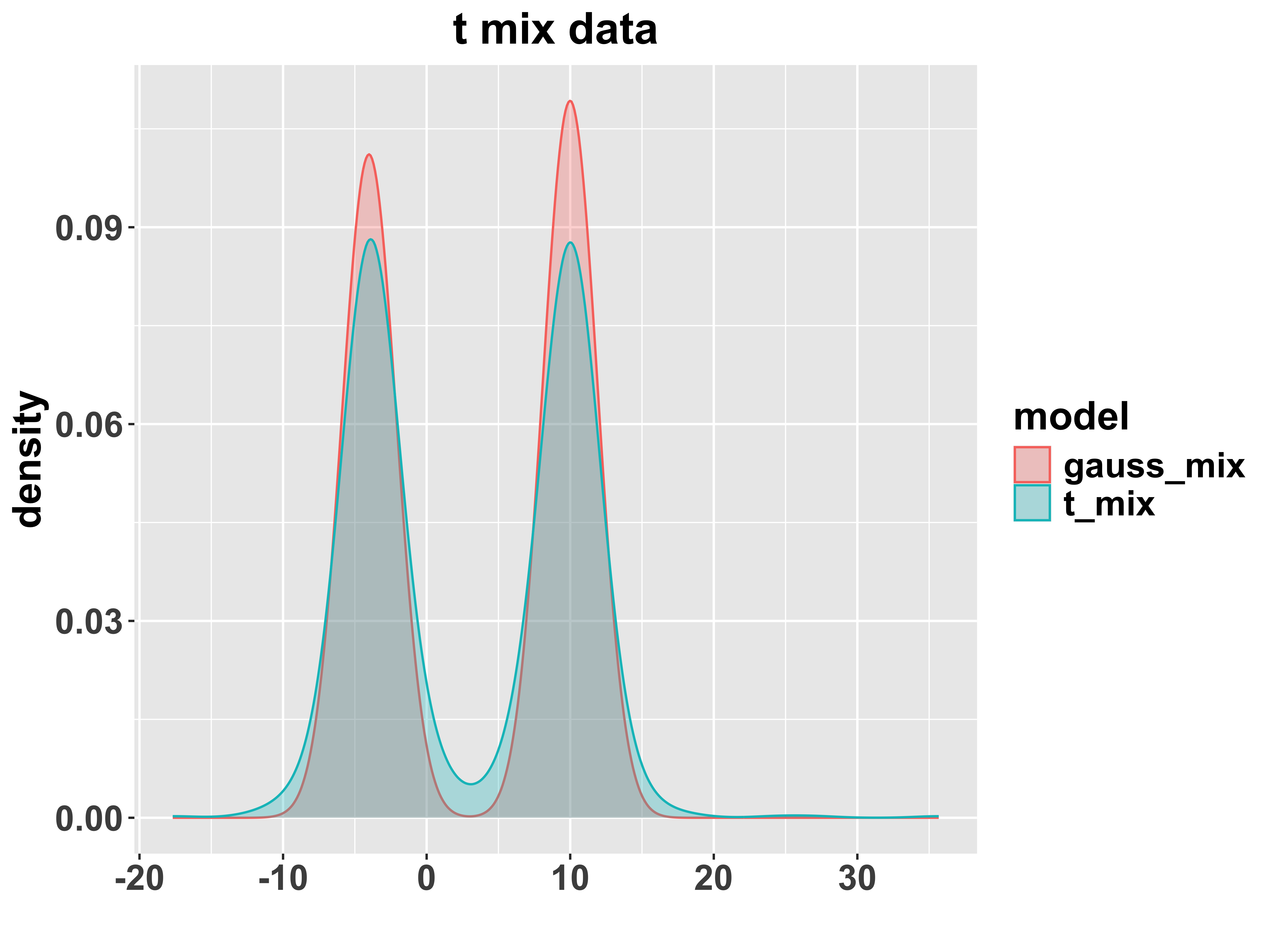}
        \caption{Data}
        \label{fig:t_mix_data}
    \end{subfigure}
    \hfill
    \begin{subfigure}[b]{0.3\textwidth}
        \centering
        \includegraphics[width=\textwidth]{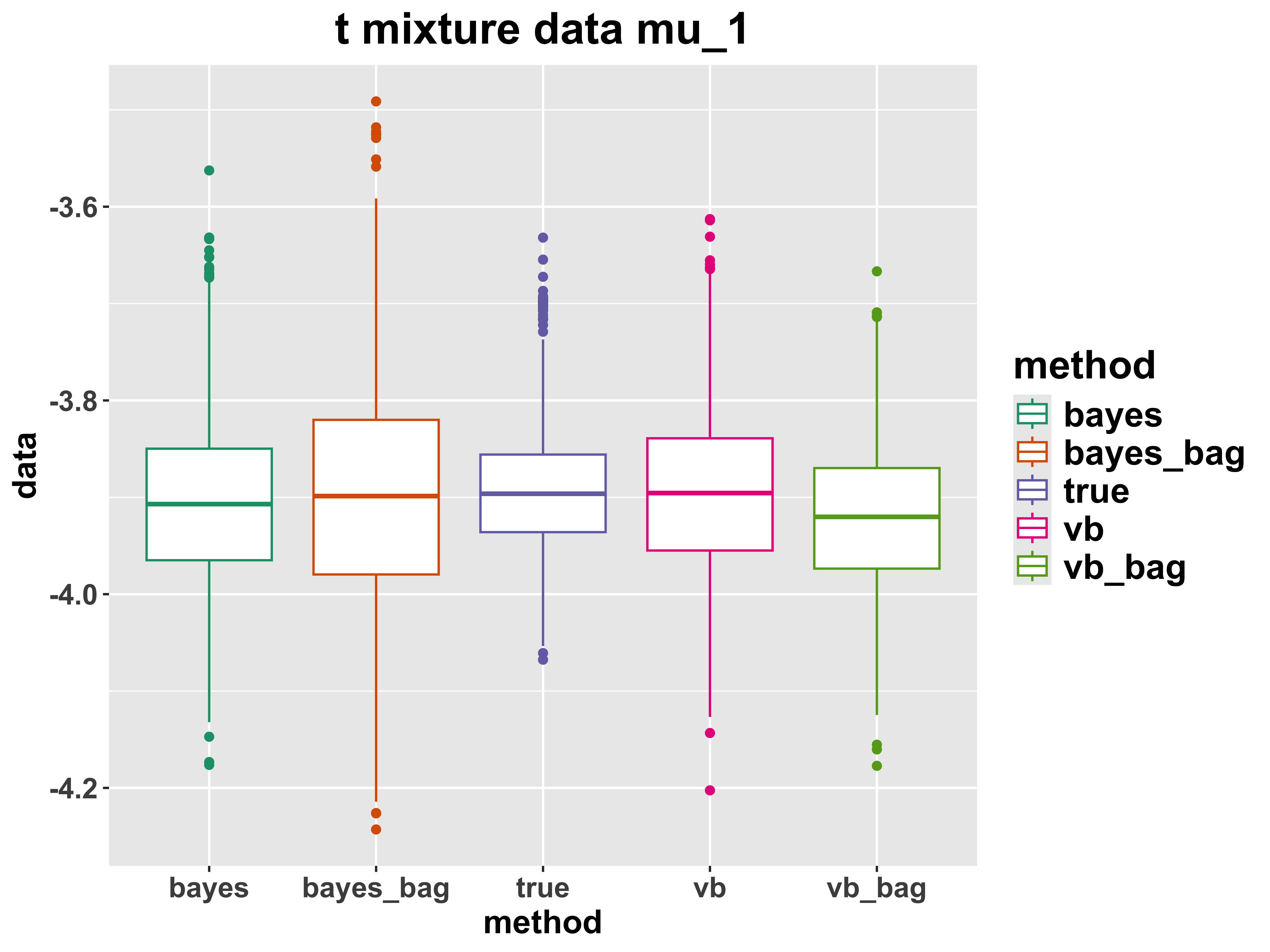}
        \caption{$\mu_1$}
        \label{fig:t_mix_mu_1}
    \end{subfigure}
    \hfill
    \begin{subfigure}[b]{0.3\textwidth}
        \centering
        \includegraphics[width=\textwidth]{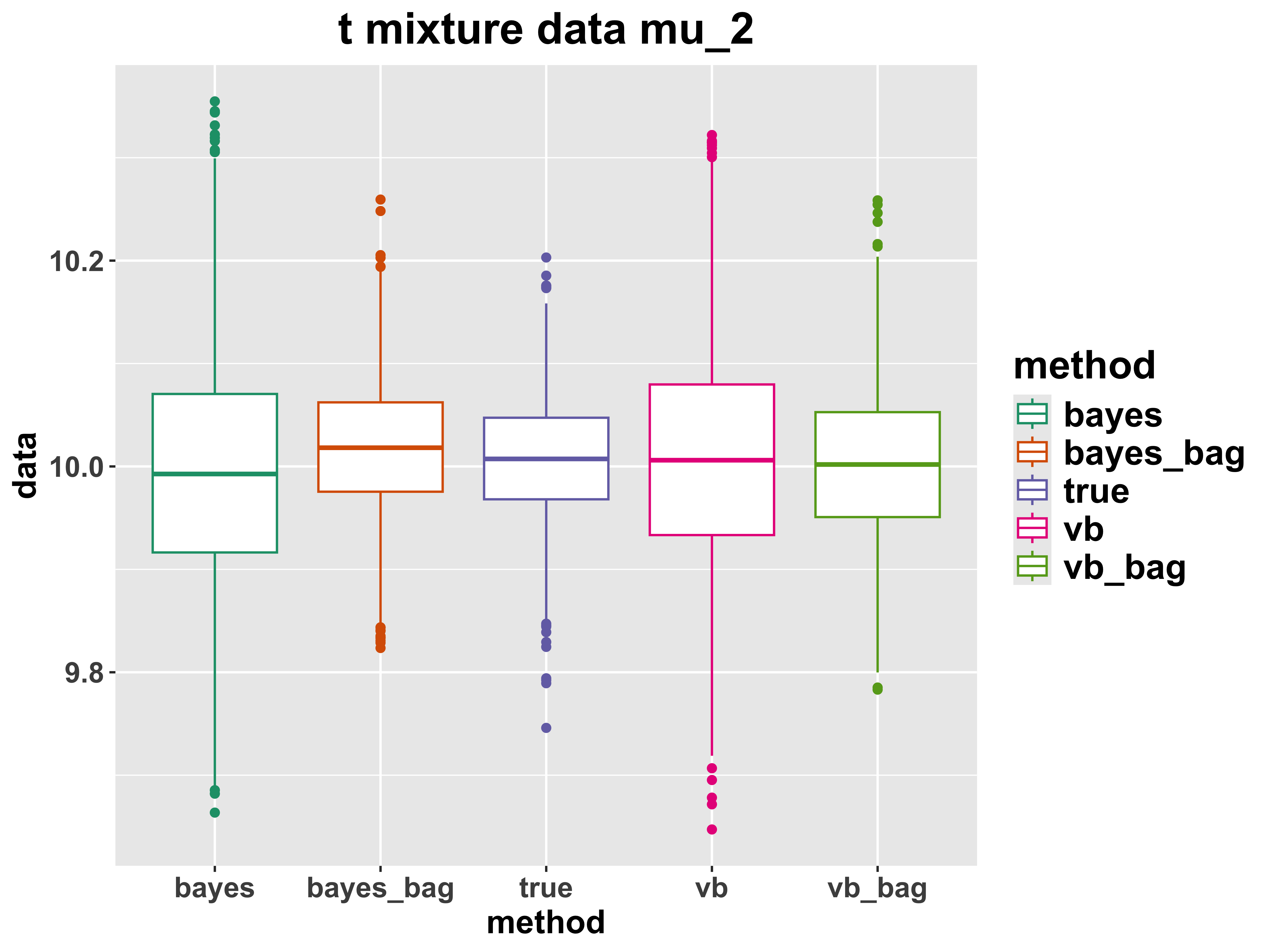}
        \caption{$\mu_2$}
        \label{fig:t_mix_mu_2}
    \end{subfigure}
    \caption{Gaussian mixture fit to $t$-mixture data.}
    \label{fig:t_mix}
\end{figure}

\begin{figure}[htbp!]
    \centering
    \begin{subfigure}[b]{0.3\textwidth}
        \centering
        \includegraphics[width=\textwidth]{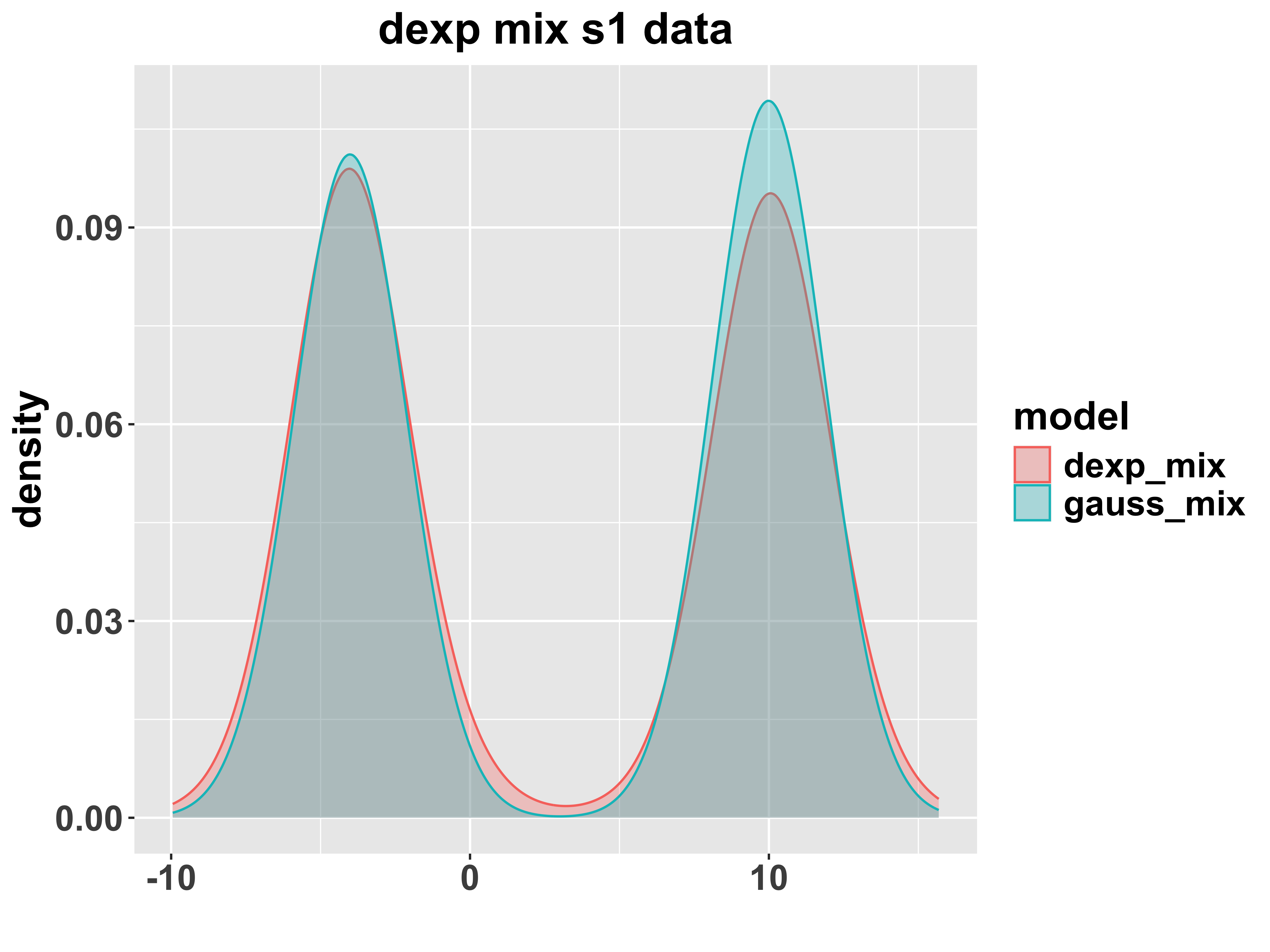}
        \caption{Data}
        \label{fig:dexp_mix_data}
    \end{subfigure}
    \hfill
    \begin{subfigure}[b]{0.3\textwidth}
        \centering
        \includegraphics[width=\textwidth]{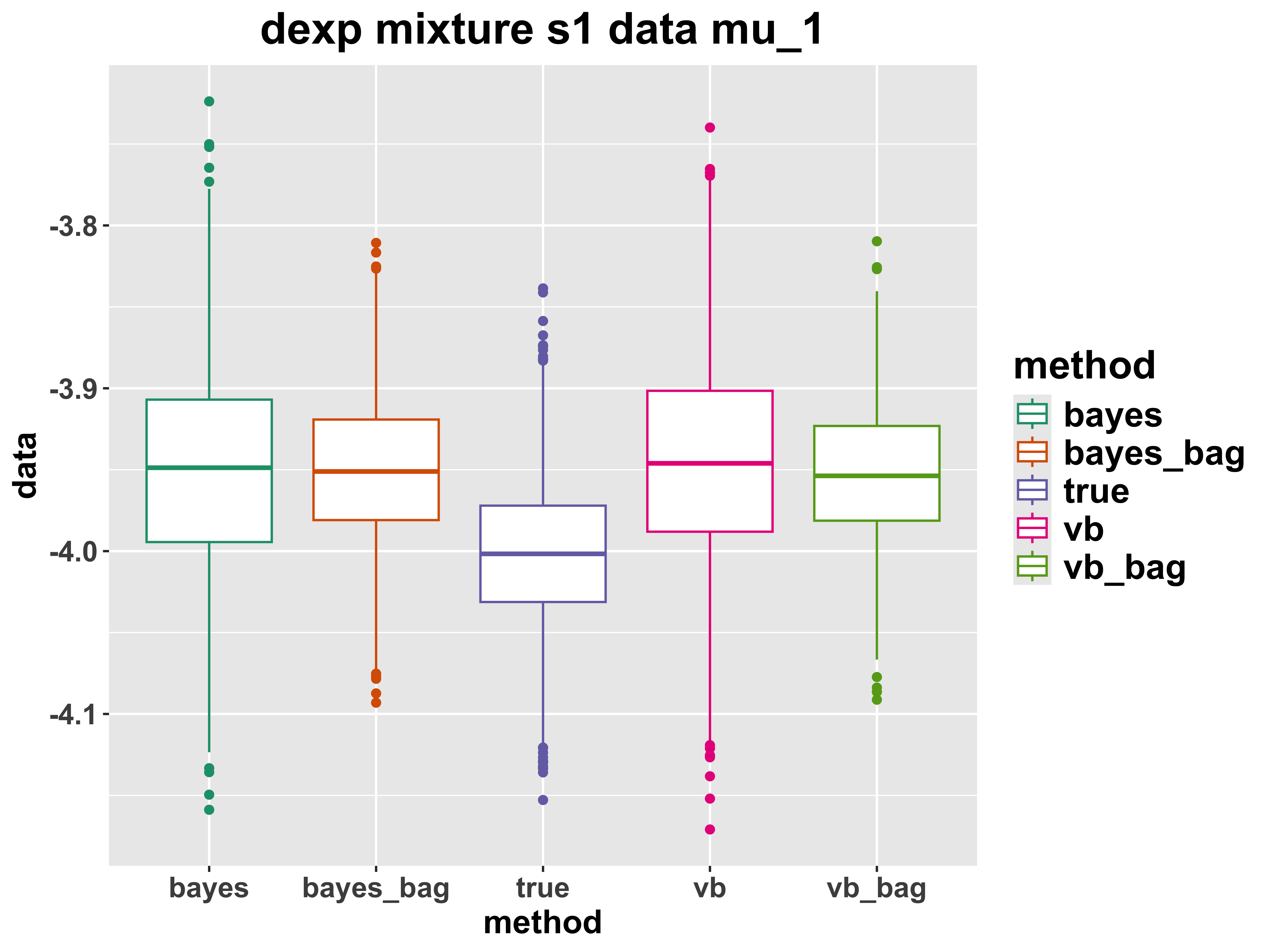}
        \caption{$\mu_1$}
        \label{fig:dexp_mix_mu_1}
    \end{subfigure}
    \hfill
    \begin{subfigure}[b]{0.3\textwidth}
        \centering
        \includegraphics[width=\textwidth]{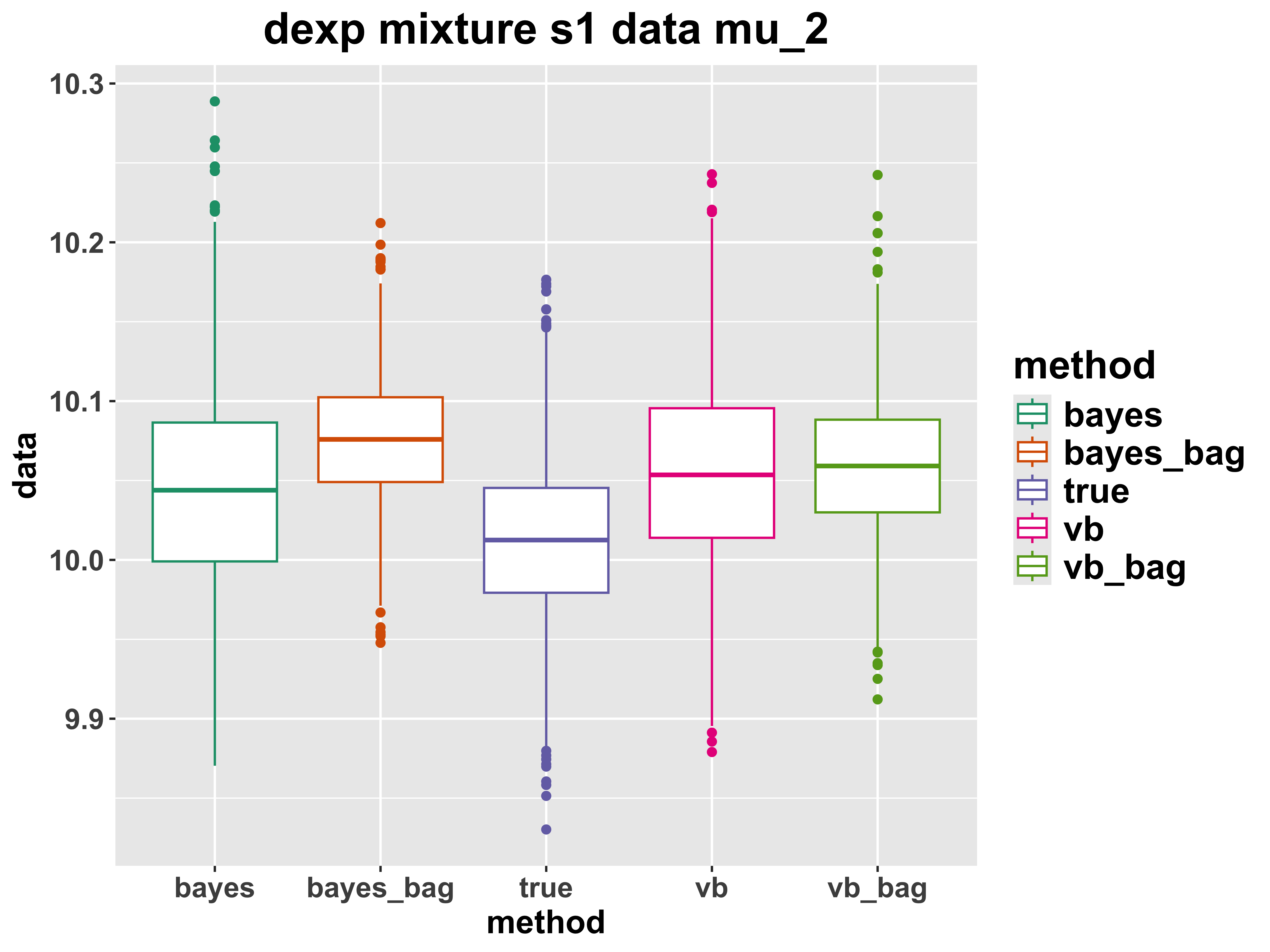}
        \caption{$\mu_2$}
        \label{fig:dexp_mix_mu_2}
    \end{subfigure}
    \caption{Gaussian mixture fit to data from a double-exponential mixture model.}
    \label{fig:dexp_mix_s1}
\end{figure}

\begin{figure}[htbp!]
    \centering
    \begin{subfigure}[b]{0.3\textwidth}
        \centering
        \includegraphics[width=\textwidth]{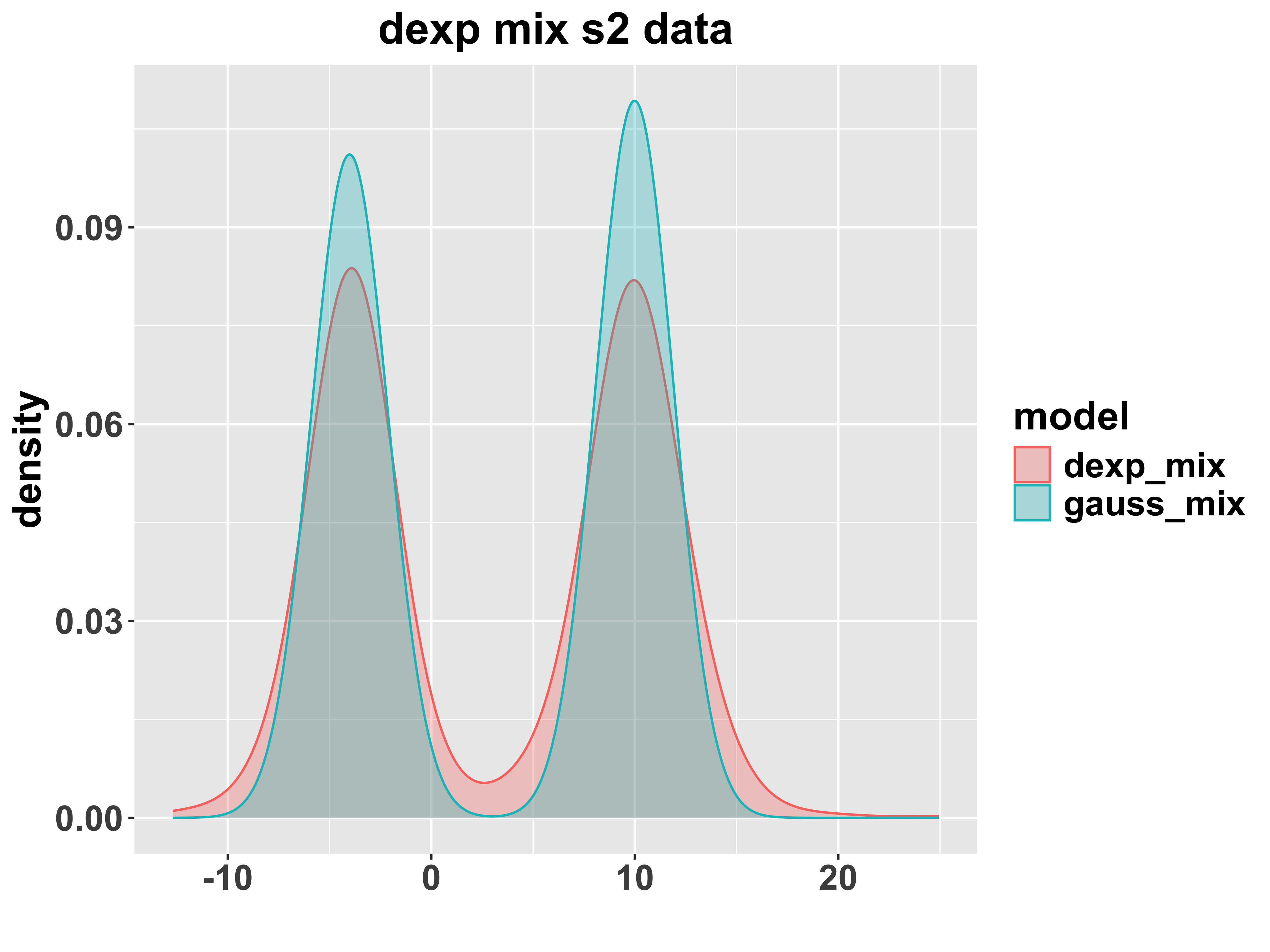}
        \caption{Data}
        \label{fig:dexp_mix_s2_data}
    \end{subfigure}
    \hfill
    \begin{subfigure}[b]{0.3\textwidth}
        \centering
        \includegraphics[width=\textwidth]{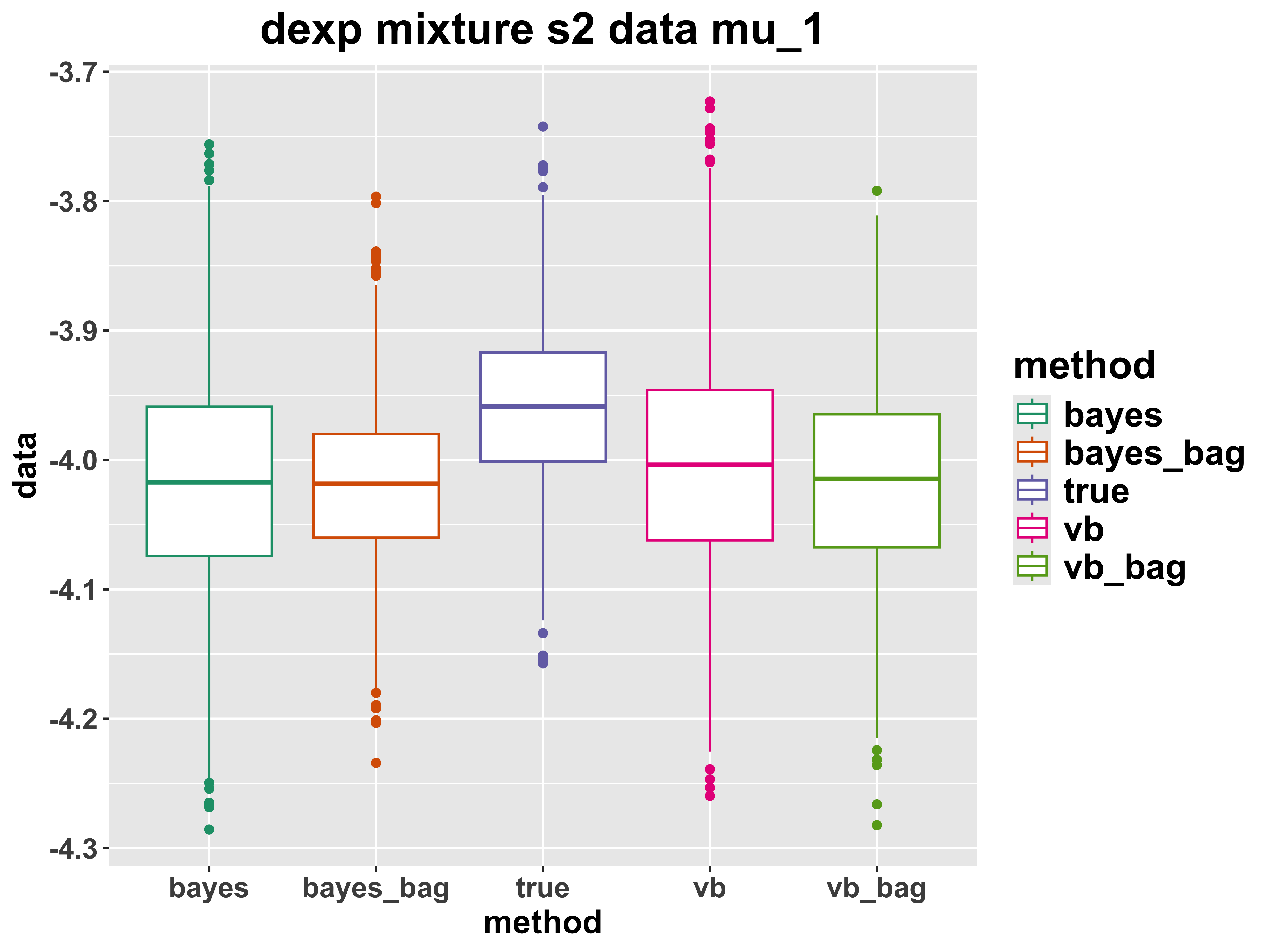}
        \caption{$\mu_1$}
        \label{fig:dexp_mix_s2_mu_1}
    \end{subfigure}
    \hfill
    \begin{subfigure}[b]{0.3\textwidth}
        \centering
        \includegraphics[width=\textwidth]{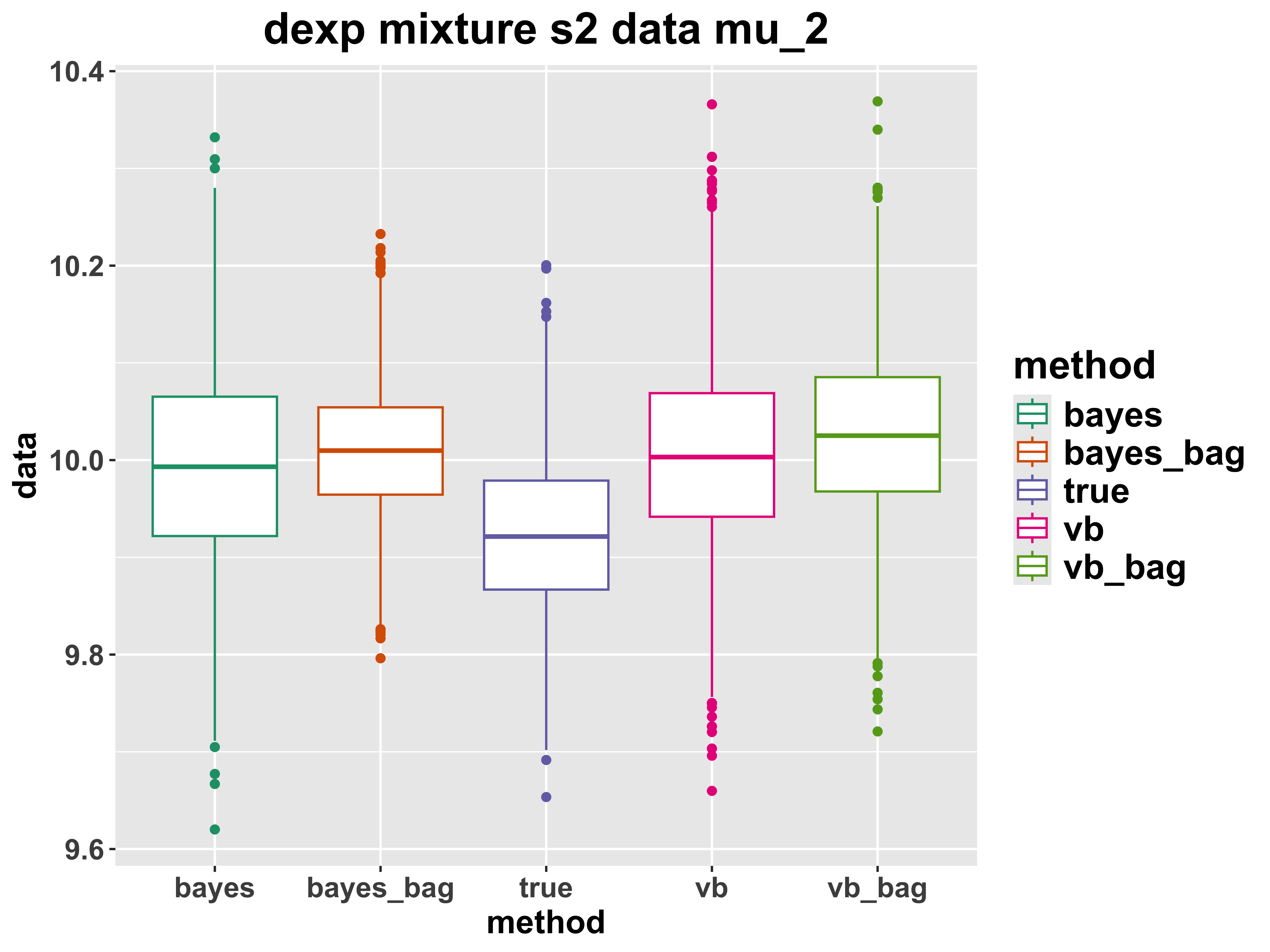}
        \caption{$\mu_2$}
        \label{fig:dexp_mix_s2_mu_2}
    \end{subfigure}
    \caption{Gaussian mixture fit to double-exponential mixture data with larger variance.}
    \label{fig:dexp_mix_s2}
\end{figure}

Figure~\ref{fig:t_mix_data} shows an example where the data are generated from a $t$-mixture distribution but fitted with a Gaussian mixture model, leading to heavier tails in the data than assumed by the working model. Figures~\ref{fig:t_mix_mu_1} and~\ref{fig:t_mix_mu_2} show results for the component means $\mu_1$ and $\mu_2$ under several approaches, including variational bagging.

When using standard Bayesian methods and mean-field variational Bayes, the posteriors tend to be reasonably centered around the true means but exhibit poorly calibrated uncertainty, with credible intervals that do not match the empirical variability induced by the heavy-tailed data. In contrast, our Variational Bagging procedure corrects for this misspecification effect and yields posterior coverage much closer to the empirical coverage. Here, “robustness’’ means that, relative to standard Bayesian or VB fits, the bagged procedures produce uncertainty statements that better reflect the true distribution of the estimators under model misspecification.

This is visually summarized in Figures~\ref{fig:t_mix_mu_1} and~\ref{fig:t_mix_mu_2}. The green boxes represent the empirical interquartile range (Q1 to Q3) of the true sampling distribution. The standard methods (olive and blue boxes) produce interquartile ranges that are noticeably misaligned with this target, either too wide or too narrow depending on the setting. By contrast, the bagging procedures (red boxes for BayesBag and purple boxes for variational bagging) yields Q1–Q3 ranges more closely aligned with the green boxes, demonstrating improved accuracy and robustness in the presence of model misspecification.

A similar phenomenon is observed when the data are generated from double-exponential mixture distributions, as shown in Figures~\ref{fig:dexp_mix_s1} and~\ref{fig:dexp_mix_s2}. In these settings as well, the bagging-based procedures effectively mitigate the impact of tail misspecification and produce uncertainty quantification that more faithfully reflects the underlying data-generating process.

Consistent with our theoretical results, the outcomes of variational bagging closely mirror those of Bayesian bagging in these simulations. This indicates that variational bagging can inherit the robustness properties of BayesBag while being substantially more computationally efficient, thereby offering a practical and effective tool for robust uncertainty quantification in mixture models and beyond.

\subsection{Sparse linear regression model}

In this section, we consider a sparse linear regression model with a spike-and-slab prior, a popular and effective prior for variable selection. The model is
\begin{align*}
    Y \mid X, \Gamma, \beta, \sigma^2 &\sim N(X \Gamma \beta, \sigma^2 I_n), \\
    \beta &\sim N(0, \sigma_\beta^2 I_q), \\
    \sigma^2 &\sim \text{InvGamma}(A,B), \\
    \gamma_i &\overset{i.i.d.}{\sim} \text{Bernoulli}(p), \quad
    \Gamma = \mathrm{diag}(\gamma_1,\ldots,\gamma_q).
\end{align*}

We simulate data from several models and fit each using this sparse linear regression specification. Unless otherwise noted, the general settings are: sample size $n = 1000$, hyperparameters $A = B = 0.1$, $\sigma_\beta^2 = 10$, and $p = 0.5$. For standard Bayesian inference, we obtain 2000 MCMC samples via Gibbs sampling. For variational inference, we use Algorithm~1 of \citet{ormerod2017variational}. For bagging, we take $B = 50$ bootstrap samples and set the bootstrap sample size to $M = \hat{M}_\infty$ in \eqref{eq-4}).

To compare methods, we focus on estimation accuracy for the regression coefficients. Let $\beta_{\mathrm{OLS}}$ be the ordinary least squares estimator (when defined), and let $\hat{\beta}$ denote a point estimate from each method (e.g., posterior mean). We compute the relative squared error (RSE)
\[
    \mathrm{RSE}
    = \frac{\|\hat{\beta} - \beta_{\mathrm{OLS}}\|^2}{\|\beta_{\mathrm{OLS}}\|^2},
\]
and compare RSEs across standard Bayes, BayesBag, variational Bayes (VB), and Variational Bagging.

We consider the following four scenarios:
\begin{itemize}
    \item Scenario 1 (S1): 10 covariates, $n = 1000$.
    \item Scenario 2 (S2): 10 covariates, $n = 2000$.
    \item Scenario 3 (S3): 20 covariates, $n = 1000$.
    \item Scenario 4 (S4): 20 covariates, $n = 2000$.
\end{itemize}

Tables~1 and~2 report RSEs under Gaussian and Student-$t$ errors, respectively.

\begin{table}[htbp!]
    \centering
    \begin{tabular}{ |c|c|c|c|c| } 
    \hline
    Scenario & Bayes & BayesBag  & VB & VB Bagging       \\
    \hline
    S1 & 3.16e-4 & 2.17e-4 & 1.24e-4 & \textbf{9.67e-5} \\
    S2 & 3.04e-4 & \textbf{2.15e-4} & 4.09e-4 & 3.63e-4 \\
    S3 & 8.82e-4 & 3.33e-4 & 1.65e-4 & \textbf{1.39e-4} \\
    S4 & 2.50e-5 & 2.37e-5 & 3.42e-5 & \textbf{2.15e-5} \\
    \hline
    \end{tabular}
    \caption{RSEs of four methods for sparse linear regression with Gaussian errors.}
\end{table}

\begin{table}[htbp!]
    \centering
    \begin{tabular}{ |c|c|c|c|c| } 
    \hline
    Scenario & Bayes & BayesBag  & VB & VB bagging      \\
    \hline
    S1 & 6.55e-4 & \textbf{3.34e-4} & 5.16e-4 & 4.14e-4 \\
    S2 & 9.73e-4 & 5.10e-4 & 1.97e-3 & \textbf{1.54e-3} \\
    S3 & 1.35e-3 & 7.80e-4 & 2.81e-4 & \textbf{2.41e-4} \\
    S4 & \textbf{6.68e-4} & 7.48e-4 & 1.15e-3 & 7.57e-4 \\
    \hline
    \end{tabular}
    \caption{RSEs of four methods for sparse linear regression with Student-$t$ errors.}
\end{table}

Across all scenarios, the bagging regimes (BayesBag and Variational Bagging) generally achieve lower RSEs than their non-bagged counterparts, highlighting the robustness benefits of bagging. Under Gaussian errors, VB bagging  often produces the most accurate results, slightly improving upon both Bayes and standard VB. Under Student-$t$ errors, where the likelihood is misspecified, the gains from bagging are even more pronounced: BayesBag consistently improves on Bayes, and Variational Bagging improves on standard VB in three of the four scenarios.

We also observe that VB can occasionally outperform standard Bayesian inference. This is likely due to Monte Carlo error and limited MCMC exploration in the Gibbs sampler, especially in higher-dimensional settings. Overall, these results indicate that bagging can enhance the accuracy and stability of both Bayesian and variational approaches for sparse linear regression, while retaining the computational advantages of VB.

\subsection{Deep learning model for prediction}

We implement variational bagging for prediction via
\begin{equation*}
    q^{\textup{bvB}}(X_{\text{new}} \mid X_{1:n})
    = \frac{1}{B} \sum_{b=1}^B q\bigl(X_{\text{new}} \mid X^{*(b)}_{1:M}\bigr),
\end{equation*}
where $X^{*(b)}_{1:M}$ denotes the $b$-th bootstrap sample. 
From Section~\ref{sec:boot_size}, the resulting asymptotic covariance for the bagged variational posterior has the form
\[
    \frac{1}{2}
    \bigl[(\tilde{V}_{\textsc{vb}}^0)^{-1}
          + (V_{\textsc{vb}}^0)^{-1}
            D_{\textsc{vb}}^0
            (V_{\textsc{vb}}^0)^{-1}\bigr],
\]
which can be viewed as a compromise between the model-based covariance $(\tilde{V}_{\textsc{vb}}^0)^{-1}$ and the sandwich covariance $(V_{\textsc{vb}}^0)^{-1} D_{\textsc{vb}}^0 (V_{\textsc{vb}}^0)^{-1}$. Even with $M = 2n$, we still recover half of the off-diagonal (sandwich) contribution through the term
\(
    \tfrac{1}{2}
    (V_{\textsc{vb}}^0)^{-1} D_{\textsc{vb}}^0 (V_{\textsc{vb}}^0)^{-1}.
\)

We apply fully-connected deep neural networks (DNNs) for regression and assess predictive uncertainty. Motivated by the fact that such DNNs can achieve near-optimal nonparametric rates \cite{ohn2024adaptive, kohler2021rate}, we investigate three regression settings:
\begin{itemize}
    \item simple linear regression,
    \item nonlinear regression,
    \item multivariate linear regression.
\end{itemize}

\begin{figure}[htbp!]
    \centering
    \begin{subfigure}[b]{0.45\textwidth}
        \centering
        \includegraphics[width=\textwidth]{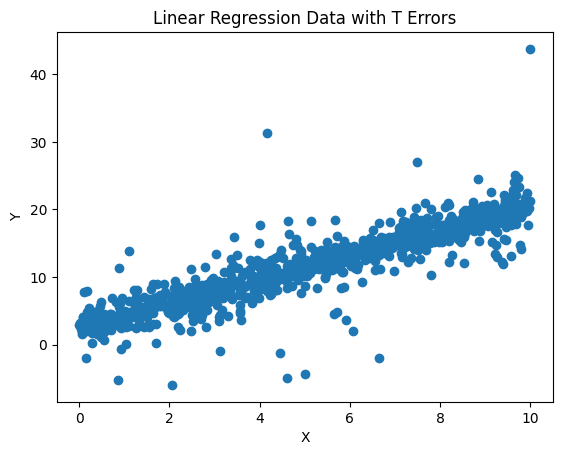}
    \end{subfigure}
    \hfill
    \begin{subfigure}[b]{0.45\textwidth}
        \centering
        \includegraphics[width=\textwidth]{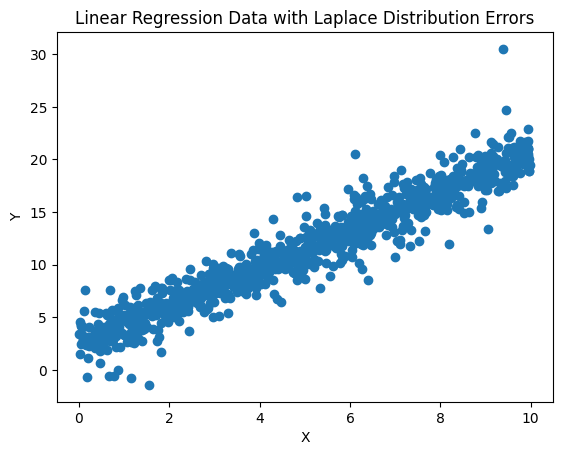}
    \end{subfigure}
    \caption{Regression data from a linear model with $t$ and Laplace errors.}
    \label{fig:linear_reg}
\end{figure}

\begin{figure}[htbp!]
    \centering
    \begin{subfigure}[b]{0.45\textwidth}
        \centering
        \includegraphics[width=\textwidth]{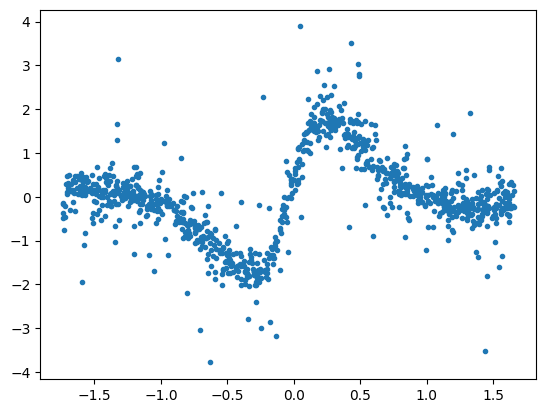}
    \end{subfigure}
    \hfill
    \begin{subfigure}[b]{0.45\textwidth}
        \centering
        \includegraphics[width=\textwidth]{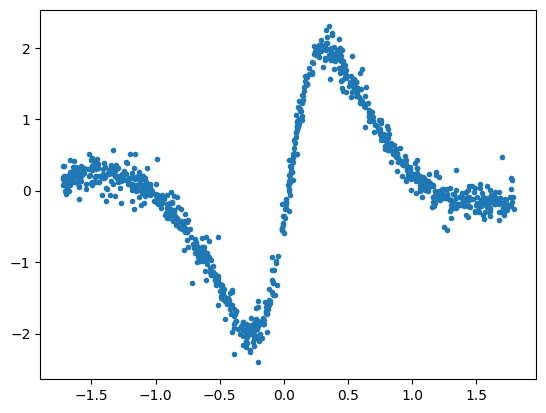}
    \end{subfigure}
    \caption{Regression data from a nonlinear model with $t$ and Laplace errors.}
    \label{fig:nonlinear_reg}
\end{figure}

For each setting, data is generated with error terms ($\epsilon$) drawn either from a Student-$t$ distribution or a Laplace distribution, but the working regression model assumes Gaussian errors, thereby introducing deliberate model misspecification.

\begin{itemize}
    \item \textbf{Simple linear regression.}
    Data are generated as
    \[
        Y = 2.5 + 1.8 X + \epsilon,
    \]
    as illustrated in Figure~\ref{fig:linear_reg}. For this setting, we use a DNN with architecture $[1, 10, 1]$ (one input, one hidden layer with 10 units, and one output), which is sufficient to capture the underlying linear relationship.

    \item \textbf{Nonlinear regression.}
    Data follow
    \[
        Y = \frac{\sin(X)}{1 + X^2} + \epsilon,
    \]
    as shown in Figure~\ref{fig:nonlinear_reg}. To capture the nonlinear structure, we employ a deeper and wider DNN with architecture $[1, 32, 64, 1]$, which provides increased capacity to approximate the nonlinear regression function.

    \item \textbf{Multivariate linear regression.}
    We generate responses via
    \[
        Y = X \beta + \epsilon,
    \]
    where $\beta_i \overset{i.i.d.}{\sim} N(0,1)$ and there are 7 predictors (so $X \in \mathbb{R}^{n \times 7}$). For this model, we use a larger DNN with architecture $[7, 128, 64, 32, 1]$ to handle the increased input dimension and capture interactions among predictors.
\end{itemize}

Our primary evaluation metric is the empirical coverage of $95\%$ predictive intervals on the original data. Specifically, we examine whether the nominal $95\%$ predictive intervals contain approximately $95\%$ of the observed responses, which serves as a diagnostic for the calibration of predictive uncertainty. In the simulation study:
\begin{itemize}
    \item the number of bootstrap replicates is $B = 10$;
    \item the bootstrap sample size is set to $M = 2n$.
\end{itemize}

Model misspecification (Gaussian working errors vs.\ heavy-tailed true errors) tends to produce predictive intervals that are too narrow, leading to under-coverage of the nominal $95\%$ predictive intervals. This effect is visible in the variational Bayes (VB) results. However, as shown in Table~\ref{tab:dnn_coverage}, variational bagging (VB bagging) substantially corrects this under-coverage and yields predictive intervals whose empirical coverage is much closer to the nominal level.

\begin{table}[htbp!]
    \centering
    \begin{tabular}{ |l|c|c| } 
    \hline
    Data setting & VB & VB bagging \\
    \hline
    linear reg + $t$ errors       & 93.8 & \bf{94.6} \\
    linear reg + Laplace errors   & 93.8 & \bf{94.6} \\
    nonlinear reg + $t$ errors    & 93.2 & \bf{95.3} \\
    nonlinear reg + Laplace errors& 92.9 & \bf{95.7} \\
    multivariate reg + $t$ errors & 93.0 & \bf{94.8} \\
    multivariate reg + Laplace errors & 93.5 & \bf{95.0} \\
    \hline
    \end{tabular}
    \caption{Empirical coverage (\%) of nominal $95\%$ predictive intervals for DNN-based regression under VB and variational bagging.}
    \label{tab:dnn_coverage}
\end{table}

Across all regression settings and both error distributions, VB alone yields slightly under-covered predictive intervals (around $93\%$ rather than $95\%$), reflecting the combination of model misspecification and variational under-dispersion. In contrast, variational bagging systematically improves coverage, bringing it close to the nominal $95\%$ level in all cases. These results highlight the robustness of variationla bagging in restoring calibrated predictive uncertainty for deep neural network models, even when the error distribution is misspecified and the underlying regression function is nonlinear or high-dimensional.

\subsection{Bagged variational autoencoder (BVAE)}

In this simulation study, we consider the application of  variational bagging to  a variational autoencoder  (VAE)  model \cite{kingma2014adam}, which we refer to as the bagged variational autoencoder.  Recall in a deep generative model where a $D$-dimensional random variable $X$  is modeled by  $X=\mathbf{f}(Z)+\boldsymbol{\epsilon}$, where $Z$  is some latent variable of dimension $d$, $\epsilon$ represents the error and $\mathbf{f}$ is the generator typically parametrized by a deep neural network. VAE is one of the primary likelihood based training methods for estimating the deep generative model, where $P(Z\mid X)$ is parameterized  by another deep neural network, the so-called encoder network.  For  implementation, we adopt the same algorithm described in \cite{chae2021likelihood}.
To model data as noisy realizations from a distribution on a 1-dimensional manifold, we let 
 $X=\mathbf{f}(Z)+\boldsymbol{\epsilon}$, where $D=2$, $\sigma_*=0$ is the true variance of the residual, and $Z$ is a univariate random variable following a standard normal distribution $N(0,1)$. We examine three different functions for the true generators $\mathbf{f}_*=\left(f_{* 1}, f_{* 2}\right)$ as:\\
    \begin{itemize}
        \item Case 1 (Figure \ref{fig:vae_data_f1}). $f_{* 1}(z)=6(z-0.5), \quad f_{* 2}(z)=0.5(z-2) z(z+2)$.\\
        \item Case 2 (Figure \ref{fig:vae_data_f2}). $f_{* 1}(z)=\cos (2 \pi z), \quad f_{* 2}(z)=\sin (2 \pi z)$.\\
        \item Case 3 (Figure \ref{fig:vae_data_f3}).\\ $\left\{\begin{array}{lll}f_{* 1}(z)=2 \cos (2 \pi z)+1, & f_{* 2}(z)=2 \sin (2 \pi z)+0.4, & \text { if } z>0.5 \\ f_{* 1}(z)=2 \cos (2 \pi z)-1, & f_{* 2}(z)=2 \sin (2 \pi z)-0.4, & \text { otherwise }\end{array}\right.$  
    \end{itemize}

\begin{figure}[htbp!]
     \centering
     \begin{subfigure}[b]{0.3\textwidth}
         \centering
         \includegraphics[width=\textwidth]{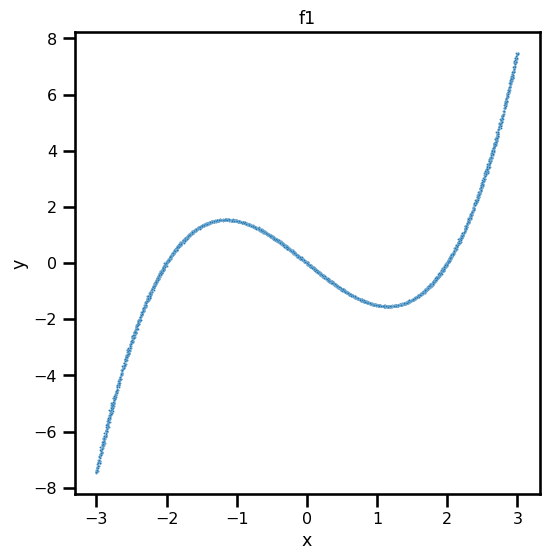}
         \caption{case 1}
         \label{fig:vae_data_f1}

     \end{subfigure}
     \hfill
     \begin{subfigure}[b]{0.3\textwidth}
         \centering
         \includegraphics[width=\textwidth]{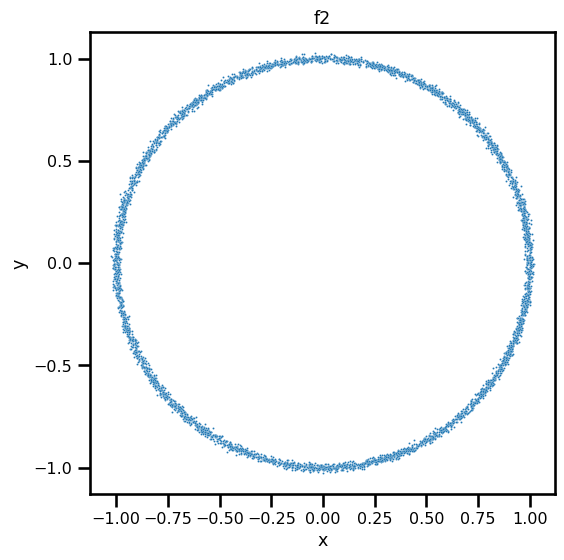}
         \caption{case 2}
         \label{fig:vae_data_f2}

     \end{subfigure}
     \hfill
     \begin{subfigure}[b]{0.3\textwidth}
         \centering
         \includegraphics[width=\textwidth]{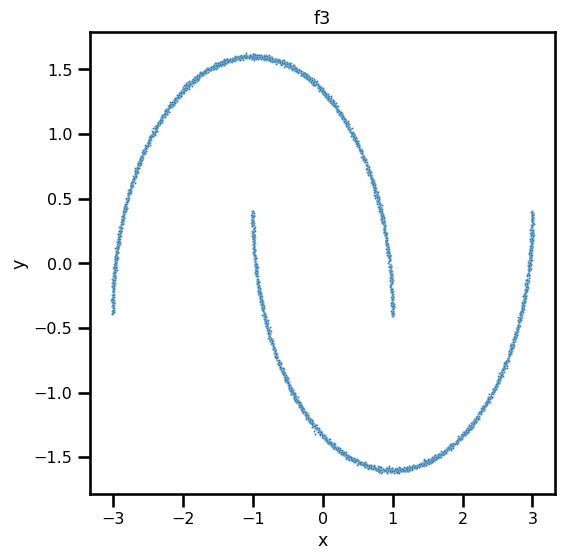}
         \caption{case 3}
         \label{fig:vae_data_f3}
     \end{subfigure}
     \caption{Simulated data concentrated on different 1-dimensional manifolds. These data will be analyzed with VAEs with and without bagging.}
     \label{fig:vae_data}
    \end{figure}
\begin{figure}[htbp!]
     \centering
     \begin{subfigure}[b]{0.3\textwidth}
         \centering
         \includegraphics[width=\textwidth]{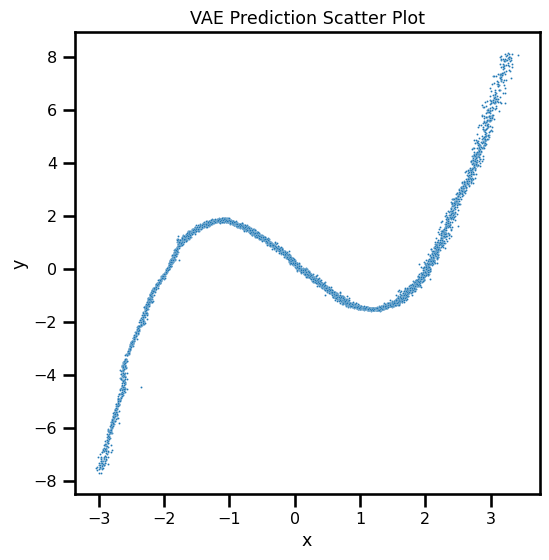}
         \caption{case 1}
         \label{fig:vae_vae_f1}

     \end{subfigure}
     \hfill
     \begin{subfigure}[b]{0.3\textwidth}
         \centering
         \includegraphics[width=\textwidth]{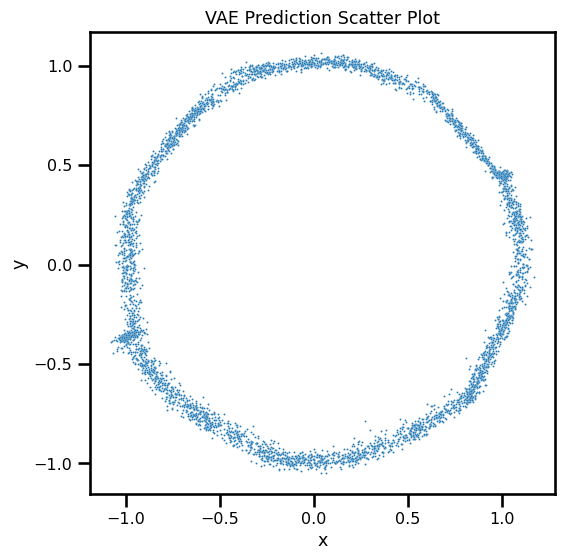}
         \caption{case 2}
         \label{fig:vae_vae_f2}
         
     \end{subfigure}
     \hfill
     \begin{subfigure}[b]{0.3\textwidth}
         \centering
         \includegraphics[width=\textwidth]{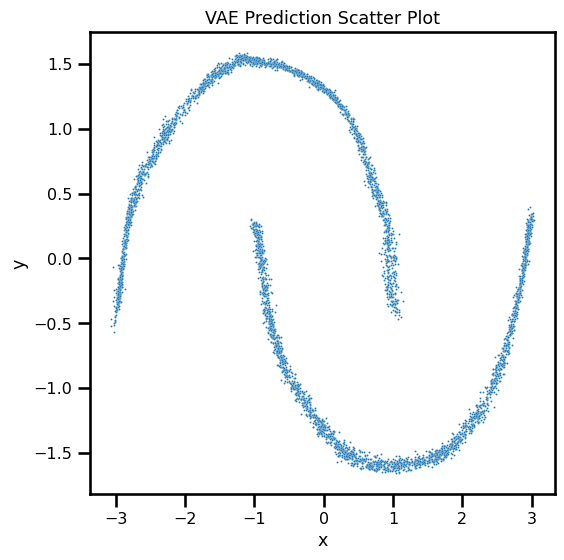}
         \caption{case 3}
         \label{fig:vae_vae_f3}         
     \end{subfigure}
     \caption{Results from apply a usual (non-bagged) VAE to the one-dimensional manifold data from the previous Figure.}
     \label{fig:vae_std}
\end{figure}

\begin{figure}[htbp!]
     \centering
     \begin{subfigure}[b]{0.3\textwidth}
         \centering
         \includegraphics[width=\textwidth]{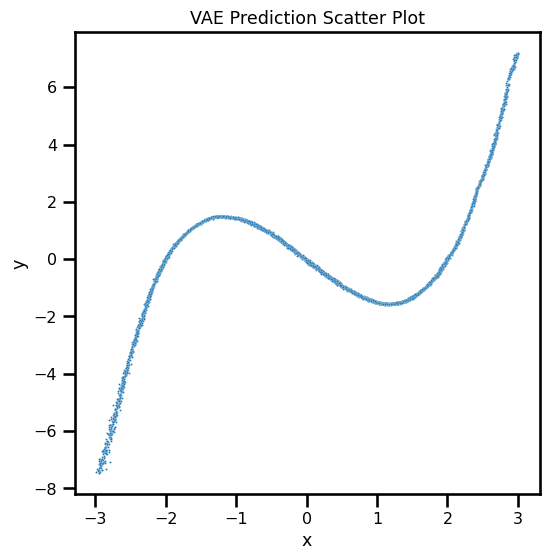}
         \caption{case 1}
         \label{fig:vae_bvae_f1}
     \end{subfigure}
     \hfill
     \begin{subfigure}[b]{0.3\textwidth}
         \centering
         \includegraphics[width=\textwidth]{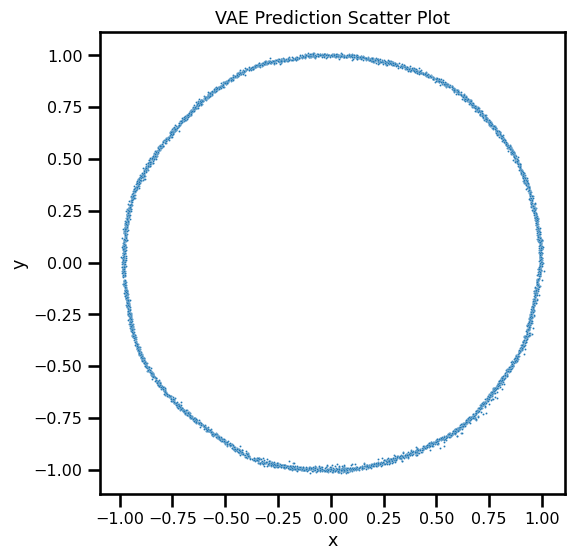}
         \caption{case 2}
         \label{fig:vae_bvae_f2}
     \end{subfigure}
     \hfill
     \begin{subfigure}[b]{0.3\textwidth}
         \centering
         \includegraphics[width=\textwidth]{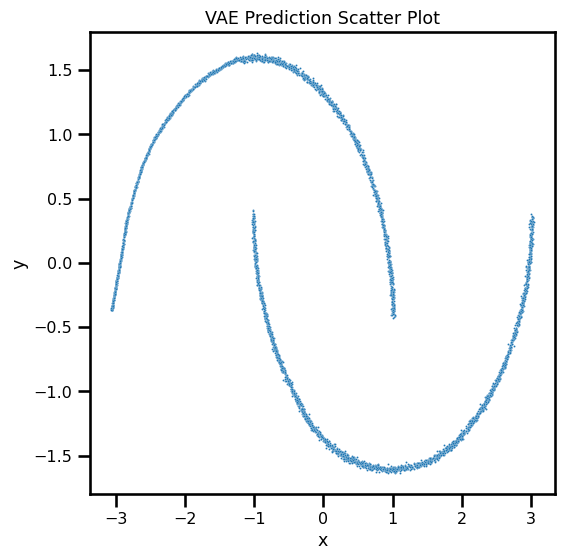}
         \caption{case 3}
         \label{fig:vae_bvae_f3}         
     \end{subfigure}
     \caption{Results of bagged VAE on the simulated one-dimensional manifold data; averaging over bootstrap dataset produces smoother and more accurate estimates}
     \label{fig:vae_bag}
    \end{figure}

Subsequently, we investigate two additional distributions: one on the Swiss roll in Figure \ref{fig:swiss} and another featuring a uniform distribution on a sphere in Figure \ref{fig:sphere}. Both distributions are supported on 2-dimensional manifolds within the ambient space $\mathbb{R}^3$. The Swiss roll distribution represents the distribution of $\mathbf{f}(Z)$, where $Z$ follows a uniform distribution on $(0,1)^2$, and the true generator $\mathbf{f}_*=\left(f_{* 1}, f_{* 2}, f_{* 3}\right):(0,1)^2 \rightarrow \mathbb{R}^3$ is defined as:
$$
\begin{aligned}
& t_1=1.5 \pi\left(1+2 z_1\right), t_2=21 z_2, \\
& f_{* 1}\left(z_1, z_2\right)=t_1 \cos \left(t_1\right), \quad f_{* 2}\left(z_1, z_2\right)=t_2, \quad f_{* 3}\left(z_1, z_2\right)=t_1 \sin \left(t_1\right) .
\end{aligned}
$$

We first use the standard VAE with  10 epochs for training.
We generate the figures by sampling from the posterior predictive distribution. We also apply our proposed bagged VAE approach. In all experiments, we set the validation and test sample sizes to 1000, while the training sample is set to 5000.
\begin{figure}[hpbp!]
     \centering
     \begin{subfigure}[b]{0.3\textwidth}
         \centering
         \includegraphics[width=\textwidth]{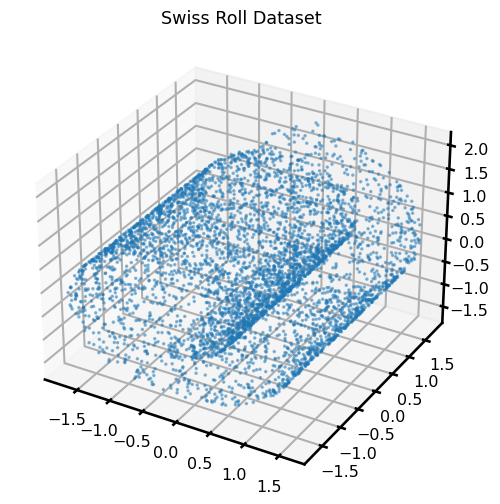}
         \caption{data}
         \label{fig:swiss_data}

     \end{subfigure}
     \hfill
     \begin{subfigure}[b]{0.3\textwidth}
         \centering
         \includegraphics[width=\textwidth]{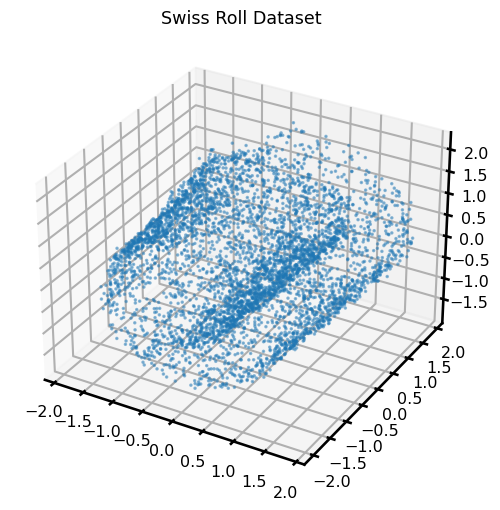}
         \caption{VAE}
         \label{fig:swiss_vae}

     \end{subfigure}
     \hfill
     \begin{subfigure}[b]{0.3\textwidth}
         \centering
         \includegraphics[width=\textwidth]{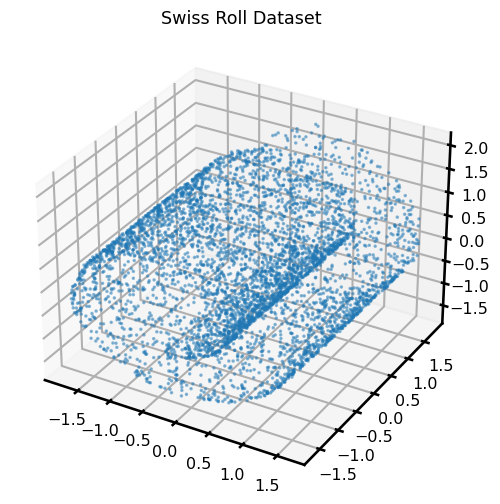}
         \caption{BVAE}
         \label{fig:swiss_bvae}
     \end{subfigure}
     \caption{Simulated data on a 2d swiss roll embedded in 3d are shown in panel (a) with the results of VAE in panel (b) and those of BVAE in panel (c).}
     \label{fig:swiss}
\end{figure}

\begin{figure}[htbp!]
     \centering
     \begin{subfigure}[b]{0.3\textwidth}
         \centering
         \includegraphics[width=\textwidth]{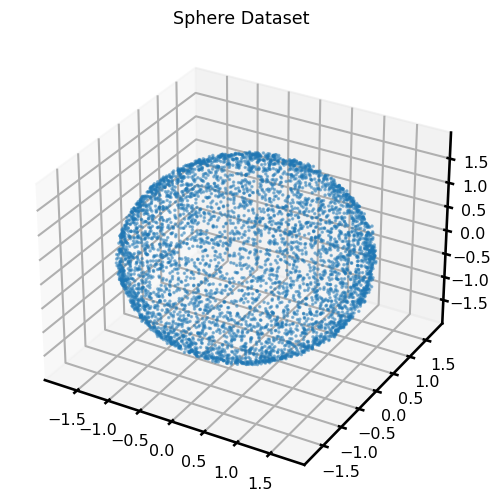}
         \caption{data}
         \label{fig:sphere_data}

     \end{subfigure}
     \hfill
     \begin{subfigure}[b]{0.3\textwidth}
         \centering
         \includegraphics[width=\textwidth]{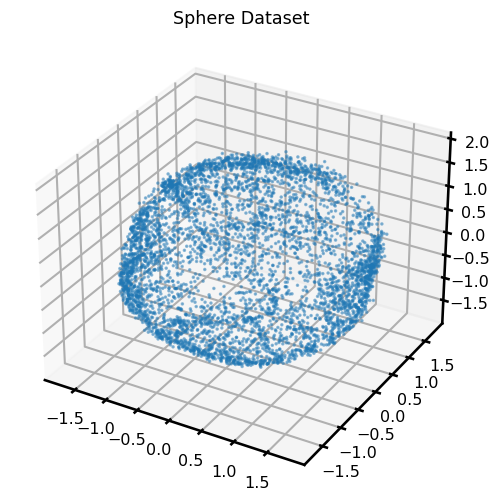}
         \caption{VAE}
         \label{fig:sphere_vae}

     \end{subfigure}
     \hfill
     \begin{subfigure}[b]{0.3\textwidth}
         \centering
         \includegraphics[width=\textwidth]{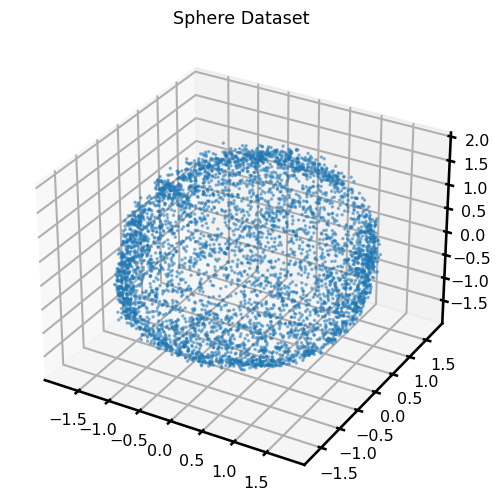}
         \caption{BVAE}
         \label{fig:sphere_bvae}
     \end{subfigure}
     \caption{Sphere data}
     \label{fig:sphere}
    \end{figure}
In our study, we have observed distinct  differences in the performance of the bagged Variational Autoencoder (BVAE) compared to the standard VAE across various data sets, especially those representing 1-dimensional manifolds. The findings can be summarized as follows:
\begin{itemize}
    \item General Observation for 1-D Manifolds in Figure \ref{fig:vae_std} and \ref{fig:vae_bag}: Across all the data sets representing 1-dimensional manifolds, the bagged VAE consistently demonstrates superior capability in reconstructing the generator. This suggests that the bagging technique enhances the VAE's ability to capture and replicate the underlying structure of the data more effectively than the standard VAE.
    \item Swiss Roll Dataset in Figure \ref{fig:swiss}: In the specific case of the Swiss roll data set, the bagged VAE again shows a notable improvement over the standard VAE. It demonstrates its ability to reconstruct the manifold structure with greater fidelity, indicating its enhanced modeling capability in this more complex data set.
    \item Sphere Dataset in Figure \ref{fig:sphere}: The sphere example presents a slightly different scenario. Here, the difference in performance between bagged VAE and standard VAE is not as pronounced. While the bagged VAE still performs marginally better in reconstructing the sphere, the improvement is more subtle then in the other cases.
\end{itemize}
These observations suggest that the bagging technique, when applied to a VAE, generally enhances the model's ability to accurately reconstruct various types of data structures, particularly in the context of 1-dimensional manifolds. The enhanced performance of the bagged VAE in most scenarios indicates its potential as a more robust and effective approach in the field of generative modeling. However, the slight improvement seen in the sphere data set also highlights that the effectiveness of bagging can vary depending on the specific characteristics of the data set being modeled.

\section{Real data analysis: Bagged VAE for the MNIST and Omniglot datasets}

We first consider the well-known MNIST dataset \citep{lecun1998gradient}. MNIST consists of grayscale images of handwritten digits of size $28 \times 28$, with a training set of 60,000 images and a test set of 10,000 images. We randomly sample 10,000 images from the training set to form a validation set.

In our VAE architecture for images, we utilize convolutional Flipout layers \citep{wen2018flipout} as variational counterparts of standard convolutional layers, and transposed convolutional Flipout layers as variational analogs of transposed convolutional layers in the decoder.

\begin{figure}[htbp!]
    \centering
    \begin{subfigure}[b]{0.3\textwidth}
        \centering
        \includegraphics[width=\textwidth]{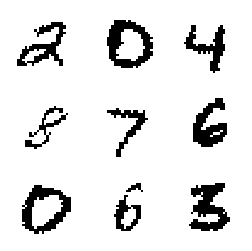}
        \caption{Original}
        \label{fig:minst_ori}
    \end{subfigure}
    \hfill
    \begin{subfigure}[b]{0.3\textwidth}
        \centering
        \includegraphics[width=\textwidth]{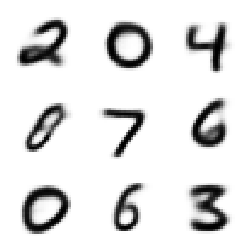}
        \caption{VAE}
        \label{fig:minst_vae}
    \end{subfigure}
    \hfill
    \begin{subfigure}[b]{0.3\textwidth}
        \centering
        \includegraphics[width=\textwidth]{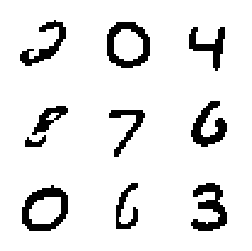}
        \caption{Bagged VAE (BVAE)}
        \label{fig:minst_bvae}
    \end{subfigure}
    \caption{MNIST: original images, standard VAE reconstructions, and bagged VAE (BVAE) reconstructions.}
    \label{fig:minst}
\end{figure}

We compare the standard VAE and the bagged VAE (BVAE) using the same protocol as in our earlier VAE simulation study. As illustrated in Figure~\ref{fig:minst}, the bagged VAE reconstructions appear visually cleaner and more representative of the underlying digit structure. In particular, the BVAE tends to denoise the images more effectively and emphasize the main strokes and shapes of each digit, while the standard VAE produces blurrier reconstructions.

Next, we consider the Omniglot dataset, which consists of handwritten character images from 50 different alphabets, each of size $28 \times 28$. The dataset contains 24,345 training samples and 8,070 test samples. As with MNIST, we split the training set into 20,000 images for training and 4,345 images for validation. We apply the same VAE and BVAE architectures and training procedures as in the MNIST experiment.

\begin{figure}[htbp!]
    \centering
    \begin{subfigure}[b]{0.3\textwidth}
        \centering
        \includegraphics[width=\textwidth]{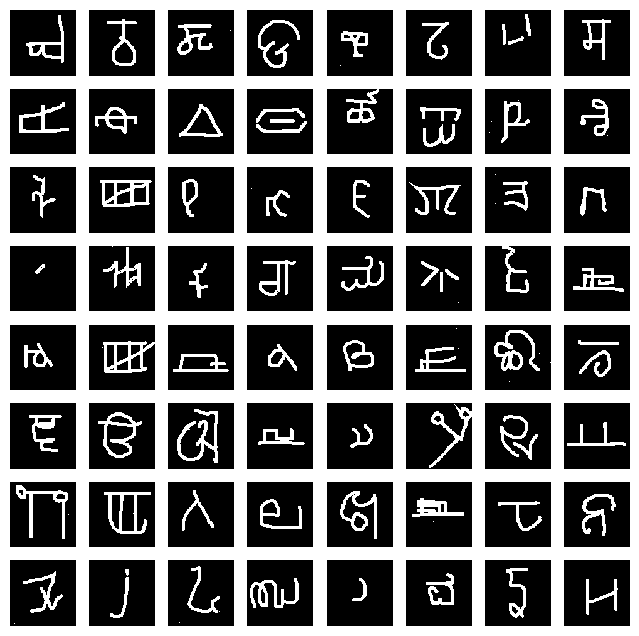}
        \caption{Original}
        \label{fig:Omniglot_ori}
    \end{subfigure}
    \hfill
    \begin{subfigure}[b]{0.3\textwidth}
        \centering
        \includegraphics[width=\textwidth]{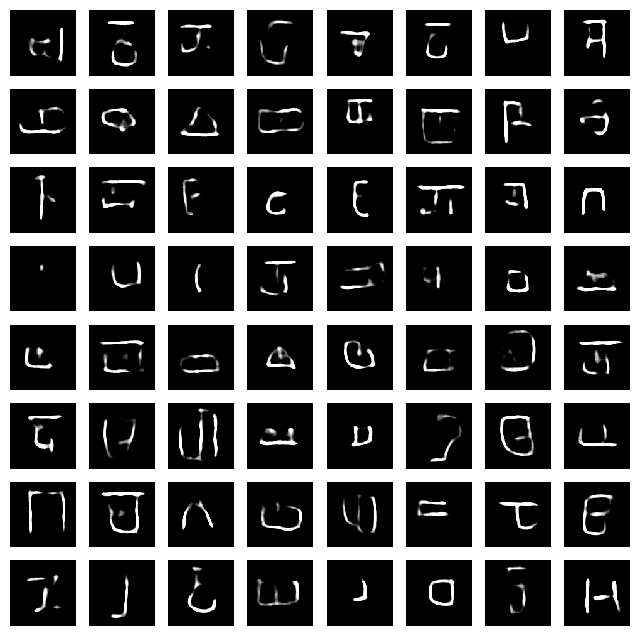}
        \caption{VAE}
        \label{fig:Omniglot_vae}
    \end{subfigure}
    \hfill
    \begin{subfigure}[b]{0.3\textwidth}
        \centering
        \includegraphics[width=\textwidth]{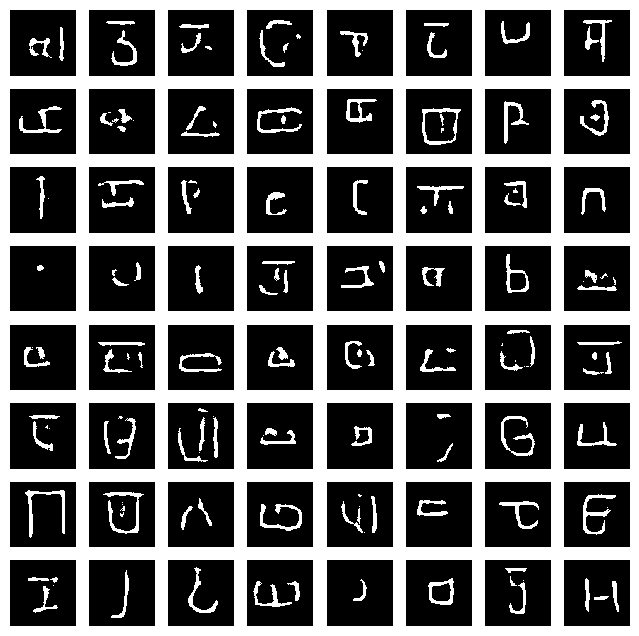}
        \caption{BVAE}
        \label{fig:Omniglot_bvae}
    \end{subfigure}
    \caption{Omniglot: original images, standard VAE reconstructions, and bagged VAE (BVAE) reconstructions.}
    \label{fig:Omniglot}
\end{figure}

Figure~\ref{fig:Omniglot} shows that the qualitative behavior observed on MNIST carries over to Omniglot. The bagged VAE reconstructions more faithfully preserve fine-grained structural details of the characters and better capture their distinctive features compared to the standard VAE. Overall, these experiments suggest that variational bagging can substantially enhance the quality and robustness of VAE-based generative modeling for image data, resulting in reconstructions that are both cleaner and more representative of the original inputs.

\appendix

\section{ Proofs}

\subsection{Proof of \cref{thm-vbbag1}}

Define $\mathbb{P}_n=n^{-1}\sum_{i=1}^n\delta_{X_i}$ and $\mathbb{G}_n = n^{1/2}(\mathbb{P}_n-P_0)$. Also, define their bootstrapped versions 
\begin{align*}
\mathbb{P}_{n}^{*}&=M^{-1} \sum_{i=1}^{n} K_{i} \delta_{X_{i}} \\
\mathbb{G}_{n}^{*}&=M^{1/2}(\mathbb{P}_{n}^{*}-\mathbb{P}_{n})
\end{align*}
where $(K_1,\dots, K_n) \sim \text{Multi}(M, 1/n)$. 
Under the conditions 1 to 4, we can apply Lemma 19.31 of \cite{van2000asymptotic} with $\ell_\theta(x)=\log{p_{\textsc{vb}}(x|\theta)}$ to have
    \begin{align}
    \label{eq:first_order_expand}
        \mathbb{G}_n\left(\sqrt{n}(\ell_{\theta_0+h_n/\sqrt{n}}-\ell_{\theta_0})-h_n^\top\dot{\ell}_{\theta_0} \right)\stackrel{P_0}{\to}0
    \end{align}
for any sequence $h_n$ bounded in $P_0$-probability. Moreover, by Theorem 23.7 of \cite{van2000asymptotic}, conditionally given $X_1,X_2,\dots$, the bootstrap empirical process $\mathbb{G}_n^*$ and the empirical process $\mathbb{G}_n$ converge weakly to the same limiting random  process, thus we also have
    \begin{align*}
        \mathbb{G}_n^*\left(\sqrt{n}(\ell_{\theta_0+h_n/\sqrt{n}}-\ell_{\theta_0})-h_n^\top\dot{\ell}_{\theta_0} \right)\stackrel{P_0}{\to}0.
    \end{align*}
This implies that
    \begin{align*}
        (nM)^{1/2}\mathbb{P}_n^*(\ell_{\theta_0+h_n/\sqrt{n}}-&\ell_{\theta_0})
        -\sqrt{n}h_n^\top (\mathbb{P}_n^*-\mathbb{P}_n)\dot{\ell}_{\theta_0}\\
        &- n\mathbb{P}_n(\ell_{\theta_0+h_n/\sqrt{n}}-\ell_{\theta_0})\stackrel{P_0}{\to}0,
    \end{align*}
where, from \labelcref{eq:first_order_expand}, the third term of the above display satisfies
    \begin{align*}
        n\mathbb{P}_n(\ell_{\theta_0+h_n/\sqrt{n}}-&\ell_{\theta_0})
        -\sqrt{n}h_n^\top (\mathbb{P}_n-P_0)\dot{\ell}_{\theta_0}\\
        &- nP_0(\ell_{\theta_0+h_n/\sqrt{n}}-\ell_{\theta_0})\stackrel{P_0}{\to}0
    \end{align*}
Moreover, by the second-order Taylar expansion in the condition 3, we have
    \begin{align*}
         -P_0(\ell_{\theta_0+h_n/\sqrt{n}}-\ell_{\theta_0})-\frac{1}{2n}h_n^{\top} V_{\textsc{vb}}^0h_n = o(1).
    \end{align*}
Therefore, letting
    \begin{align*}
         \Delta_{n}&=    n^{1 / 2}(V_{\textsc{vb}}^0)^{-1}\left(\mathbb{P}_{n}-P_0\right) \dot{\ell}_{\theta_0}\\
        \Delta_{n}^{*}&=n^{1 / 2} (V_{\textsc{vb}}^0)^{-1}\left(\mathbb{P}_{n}^{*}-\mathbb{P}_{n}\right) \dot{\ell}_{\theta_0},
    \end{align*}
we have for every compact $K \subset \Theta$,
\begin{equation}
    \sup _{h \in K}\left|M \mathbb{P}_{n}^{*}(\ell_{\theta_0+h_n/\sqrt{n}}-\ell_{\theta_0})-h^{\top}\left(cV_{\textsc{vb}}^0\right)\left(\Delta_{n}+\Delta_{n}^{*}\right)-\frac{1}{2} h^{\top}\left(cV_{\textsc{vb}}^0\right) h\right| \stackrel{P_{+}}{\rightarrow} 0.
\end{equation}
Then we can apply the second result of Theorem 3 of \cite{wang2019variational} with $X^*_{1:M}$ in place of $X_{1:n}$, $ \mathbb{P}_{n}^{*}$ in place of $\mathbb{P}_{n}$, $cV_{\textsc{vb}}^0$ in place of $V_{\theta_0}$, and $\Delta_{n}+\Delta_{n}^{*}$ in place of $\Delta_{n,\theta_0}$, to obtain that
    \begin{align}
        \left\|\mathcal{L}(\sqrt{n}(\vartheta'-\theta_0)|X^*_{1:M})-N(\Delta_{n}+\Delta_{n}^{*},(c\tilde{V}_{\textsc{vb}}^0)^{-1})\right\|_{TV}\stackrel{P_{0}}{\rightarrow} 0
    \end{align}
for $\vartheta'\sim q(\theta|X^*_{1:M})$, where $\mathcal{L}(\sqrt{n}(\vartheta'-\theta_0)|X^*_{1:M})$ denotes the conditional raw of $\sqrt{n}(\vartheta'-\theta_0)$ given $X^*_{1:M}$. Hence, the characteristic function of $\sqrt{n}(\vartheta^\dag-\theta_0)-\Delta_{n}\mid X_{1: n}$  for $\vartheta^\dag\sim q^{\textup{bvB}}(\theta|X_{1:n})$ evaluated at $t \in \mathbb{R}^{d}$ can be written as
\begin{align}
\label{eq:chf_bvp}
E\left[\exp \left\{i (\Delta_{n}^{* })^\top t-\frac{1}{2c}t^{\top}(\tilde{V}_{\textsc{vb}}^0)^{-1} t \right\} \mid X
_{1: n}\right]+\epsilon_{n}(t)
\end{align}
for some function $\epsilon_{n}(t)$ such that $\limsup_{n\to\infty}\sup_t\epsilon_n(t)=0.$
We can expand \labelcref{eq:chf_bvp} as
\begin{align*}
    &E \left[ \exp \left\{ i n^{1/2}\mathbb{P}_{n}^{*}\dot{\ell}_{\theta_0}(V^0_{\textsc{vb}})^{-1}t  \right\}| X_{1: n}\right] \exp \left\{ -i n^{1/2}\mathbb{P}_{n}\dot{\ell}_{\theta_0}(V^0_{\textsc{vb}})^{-1}t   \right\}\\
    & \times \exp\left\{-t^\top(\tilde{V}^0_{\textsc{vb}})^{-1}t/2c\right\}+\epsilon_n(t)
\end{align*}
The first line can be written as:
\begin{align*}
      &E \left[ \exp \left\{ i n^{1/2} M^{-1}
      \sum_{j=1}^n K_j \dot{\ell}_{\theta_0}(X_j)^\top(V^0_{\textsc{vb}})^{-1}t  \right\}|X_{1: n}\right]\\
      &\times\exp \left\{ -i n^{1/2}\mathbb{P}_{n}\dot{\ell}_{\theta_0}(V^0_{\textsc{vb}})^{-1}t   \right\}\\
    &=\left[\frac{1}{n} \sum_{j=1}^n \exp\left\{ \frac{i n^{1/2} \dot{\ell}_{\theta_0}(X_j)^\top(V^0_{\textsc{vb}})^{-1}t}{M}\right\} \right]^{M} \exp \left\{ -i n^{1/2}\mathbb{P}_{n}\dot{\ell}_{\theta_0}(V^0_{\textsc{vb}})^{-1}t   \right\}\\
    &=   \left[\frac{1}{n} \sum_{j=1}^n \exp\left\{ \frac{i n^{1/2} \delta\dot{\ell}_{\theta_0}(X_j)^\top(V^0_{\textsc{vb}})^{-1}t}{M}\right\} \right]^{M}
\end{align*}
where we let $\delta \dot{\ell}_{\theta_0}(X_j) =\dot{\ell}_{\theta_0}(X_j) - \mathbb{P}_{n}\dot{\ell}_{\theta_0}$. By the second-order Taylor expansion, the last display can be further written as
    \begin{align*}
          \left[1-\frac{n \mathbb{P}_n[ (\delta\dot{\ell}_{\theta_0})( \delta\dot{\ell}_{\theta_0})^\top](V^0_{\textsc{vb}})^{-1}t}{2M^2}+R_n \right]^{M},
    \end{align*}
where $R_n=O_{P}(1/M^3)$ denotes the remainder term. Then since $$\mathbb{P}_n[ (\delta\dot{\ell}_{\theta_0})( \delta\dot{\ell}_{\theta_0})^\top]\to D_{\textsc{vb}}^0$$ almost surely and $M/n\to c$ by assumption, the last display converges to 
\begin{equation}
    \exp \left\{-\frac{1}{2c}t^{\top}(\tilde{V}_{\textsc{vb}}^0)^{-1} t
    -\frac{1}{2c}t^{\top}(V_{\textsc{vb}}^0)^{-1} D_{\textsc{vb}}^0(V_{\textsc{vb}}^0)^{-1} t \right\}
\end{equation}
This implies the desired
    \begin{align*}
       \sqrt{n}(\vartheta^{\dag}-\theta_0) -\Delta_n\mid X_{1: n}
        \stackrel{d}{\to}N(0,(\tilde{V}_{\textsc{vb}}^0)^{-1}/c+(V_{\textsc{vb}}^0)^{-1} D_{\textsc{vb}}^0(V_{\textsc{vb}}^0)^{-1}/c)
    \end{align*}
by Levy’s continuity theorem.

\subsection{Proof of \cref{thm:confidence}}

Since $\hat{\theta}_{\textup{mle}}\to\theta_0$ and $\Delta_n\to0$ in $P_0$-prabability as well as $\widehat{\Sigma}\to \Sigma_0$ in $P_0$-prabability by assumption, by \cref{thm:bvm_no_latent}, we have 
    \begin{align*}
      P_{N(0,\Sigma_0) }( C(r_{n, 1-\alpha})) &= Q^{\textup{bvB}}(\vartheta^\dag\in C(r_{n, 1-\alpha})) + o_{P_0}(1)\\
      &=1-\alpha+ o_{P_0}(1),
    \end{align*}
where we denote by $ P_{N(0,\Sigma) }$ the probability measure under the normal distribution $N(0, \Sigma)$. Therefore by the continuous mapping theorem, $r_{n,1-\alpha}^2\to \chi^2_{d, 1-\alpha}$ in $P_0$-probability, where $ \chi^2_{d, 1-\alpha}$ denotes the $1-\alpha$ quntitle of the $\chi^2(d)$ distribution and $d$ denotes the dimension of the parameter $\theta$.  We use this fact to get, letting $S_0=(V^0)^{-1}D^0(V^0)^{-1}$ for notational simplicity,
    \begin{align*}
       P_0(\theta_0\in C(r_{n,1-\alpha}))
     &  =P_0( \sqrt{n}(\hat{\theta}_{\textup{mle}}-\theta_0)\in \{u: u^\top\Sigma_0^{-1}u\le r_{n, 1-\alpha}^2\})+o(1)\\
      & =P_0(\sqrt{n}(\hat{\theta}_{\textup{mle}}-\theta_0)\in \{u: u^\top\Sigma_0^{-1}u\le \chi^2_{d, 1-\alpha}\})+o(1)\\
      &=P_{U\sim N(0, S_0)}(U^\top\Sigma_0^{-1}U\le \chi^2_{d, 1-\alpha})+o(1)
    \end{align*}
where the last equality holds due to
    \begin{align*}
       \sqrt{n}(\hat{\theta}_{\textup{mle}}-\theta_0)
       \overset{d}{\rightarrow} N(0,S_0).
    \end{align*}
Therefore, since $\Sigma_0- S_0=(\tilde{V}^0)^{-1}$ is non-negative definite, so is $S_0^{-1}-\Sigma_0^{-1}$, we have 
    \begin{align*}
       & P_{U\sim N(0, S_0)}( U^\top\Sigma_0^{-1}U\le \chi^2_{d, 1-\alpha})\\
        &\ge P_{U\sim N(0, S_0)}(U^\top S_0^{-1}U\le \chi^2_{d, 1-\alpha})=1-\alpha,
    \end{align*}
which completes the proof.

\subsection{Proof of \cref{thm:contraction}}

Before giving the proof, we introduce additional notations. \\Let $p_{\theta, K_{1:n}}(X_{1:n})=\prod_{i=1}^n p(X_i|\theta)^{K_i}$ and $p_{0, K_{1:n}}(X_{1:n})=\prod_{i=1}^n p(X_i|\theta)^{K_i}$ be denote the density of the bootstrapped sample $X_{1:n}^*$ with the ``bootstrap weight'' $K_{1:n}=(K_1,\dots, K_n)\sim\text{Multi}(n, 1/n)$, under $P_\theta$ and $P_0$, respectively. Then the posterior given the bootstrapped sample can be written as $\pi(\theta |X_{1:n}^*)=p_{\theta, K_{1:n}}(X_{1:n})\pi(\theta)/p_{\Pi, K_{1:n}}(X_{1:n})$ with $p_{\Pi, K_{1:n}}(X_{1:n})=\int\prod_{i=1}^n p(X_i|\theta)^{K_i}\pi(\theta) d\theta$. In what follows, the probability measure $P$ and the expectation operator $E$ are taken on the randomness of both the sample $X_{1:n}$ and the bootstrap weights $K_{1:n}$.

Let $Q^*$ be the variational posterior obtained with a bootstrapped sample $X_{1:n}^*$. The proof is done if we show that the contraction rate of $Q^*$ is given by $\epsilon_n\sqrt{\log n}$. We start with observing that, by Donsker and Varadhan’s variational inequality (e.g., Lemma J.1 of \cite{ohn2024adaptive}), 
    \begin{align}
    \label{eq:dv_ineq}
        Q(A)\le \frac{1}{t}\{\textup{KL}(Q,\Pi(\cdot|X_{1:n}^*)) +e^{t} \Pi(A|X_{1:n}^*) \}
    \end{align}
for any distribution $Q$ on $\Theta$, any event $A\subset \Theta$ and any positive number $t$. We apply the above display to $Q=Q^*$ and $A=A_n=\{\theta\in \Theta:H^2(P_{\theta},P_{0})\ge \eta_n^2\}$ with $\eta_n^2=M_n\epsilon_n^2\log n$.

Next, we define the event
    \begin{align*}
        G_n=\left\{\int\frac{p_{\theta, K_{1:n}}(X_{1:n})}{p_{0, K_{1:n}}(X_{1:n})}\pi(\theta) d\theta\ge \exp(-(C_1+2)n\epsilon_n^2)\right\}.
        \end{align*}
We have
    \begin{align*}
       E\left[\log\frac{p_{\theta, K_{1:n}}(X_{1:n})}{p_{0, K_{1:n}}(X_{1:n})}\right]=n\textup{KL}( P_0, P_\theta)
    \end{align*}
and
    \begin{align*}
        Var\left(\log\frac{p_{\theta, K_{1:n}}(X_{1:n})}{p_{0, K_{1:n}}(X_{1:n})}\right)
        &\le\sum_{i=1}^nE\left[\left(K_i\log \frac{p(X_i|\theta)}{p_0(X_i)}\right)^2\right]\\
        &=\sum_{i=1}^nE(K_i^2) \text{KLV}(P_0, P_\theta)\\
        &\le 2n\text{KLV}(P_0, P_\theta),
    \end{align*}
where the equality follows from the independence of $K_i$ and $X_i$. Hence, by a standard argument for bounding the probability of $G_n$ under the prior mass condition (e.g., Lemma 8.10 of \cite{ghosal2017fundamentals}), we have  $P(G_n^c)\le 2/(n\epsilon_n^2)\to0$.  Therefore, it suffices to bound the quantity $E(Q^*(A_n)\mathbb{I}(G_n))$, which, from the inequality in \labelcref{eq:dv_ineq}, is further bounded by
    \begin{align}
   & E(Q^*(A_n)\mathbb{I}(G_n))\nonumber\\
        &\le  \frac{1}{t}\{E[\textup{KL}(Q^*,\Pi(\cdot|X_{1:n}^*))\mathbb{I}(G_n)] 
        +e^{t}E[ \Pi(A_n|X_{1:n}^*)\mathbb{I}(G_n)] \}\label{eq:vi_ineq}.
    \end{align}
For the first term in  \labelcref{eq:vi_ineq}, we have
    \begin{align*}
    &E[\textup{KL}(Q^*,\Pi(\cdot|X_{1:n}^*))\mathbb{I}(G_n)] 
    \le E[  \textup{KL}(Q^*,\Pi(\cdot|X_{1:n}^*) ]\\
      &= E\left[ \int \left\{ \log \frac{q^*(\theta)}{\pi(\theta)}+\log\frac{p_{0, K_{1:n}}(X_{1:n})}{p_{\theta, K_{1:n}}(X_{1:n})}\right\}dQ^*(\theta) +\log\frac{p_{\Pi, K_{1:n}}(X_{1:n})}{p_{0, K_{1:n}}(X_{1:n})}\right] 
      \\
      &\le  E\left[ \int \left\{ \log \frac{q(\theta)}{\pi(\theta)}+\log\frac{p_{0, K_{1:n}}(X_{1:n})}{p_{\theta, K_{1:n}}(X_{1:n})}\right\}dQ(\theta) +\log\frac{p_{\Pi, K_{1:n}}(X_{1:n})}{p_{0, K_{1:n}}(X_{1:n})}\right] \\
      &=\textup{KL}(Q, \Pi)+n \int \textup{KL}(P_\theta,P_0)d Q(\theta)+  E\left[ \log\frac{p_{\Pi, K_{1:n}}(X_{1:n})}{p_{0, K_{1:n}}(X_{1:n})}\right] 
    \end{align*}
for any $Q\in \mathcal{Q}$, where we use the optimization optimality of $Q^*$ in the third line and use the assumption   $\sum_{i=1}^n K_i=n$ in the last line. Moreover, by Jensen's inequality together with that $\sum_{i=1}^n K_i=n$, we have
    \begin{align*}
         E\left[ \log\frac{p_{\Pi, K_{1:n}}(X_{1:n})}{p_{0, K_{1:n}}(X_{1:n})}\right] 
        & = E\left[ \log\frac{p_{\Pi, K_{1:n}}(X_{1:n})}{\prod_{i=1}^np_{0}(X_{i})}+\log\frac{\prod_{i=1}^np_{0}(X_{i})}{p_{0, K_{1:n}}(X_{1:n})}\right] \\
         &\le E\left[\log\frac{\prod_{i=1}^np_{0}(X_{i})}{p_{0, K_{1:n}}(X_{1:n})}\right]\\
         &= E\left[\sum_{i=1}^n(1-K_i)\log p_0(X_i)\right]=0.
    \end{align*}
Therefore, by our variational family assumption, 
    \begin{align}
        E[  \textup{KL}(Q^*,\Pi(\cdot|X_{1:n}^*) ]\le C_1' n\epsilon_n^2
    \end{align}
for some constant $C_1'>0$. 

Now we focus on the second term in  \labelcref{eq:vi_ineq}. We can bound it as
    \begin{align*}
        \Pi(A_n|X_{1:n}^*)\mathbb{I}(G_n)
        \le \Pi(A_n\cap \Theta_n|X_{1:n}^*)\mathbb{I}(G_n)+ \Pi( \Theta_n^c|X_{1:n}^*)\mathbb{I}(G_n).
    \end{align*}
The expectation of the second term in the above display satisfies
    \begin{align*}
        E[\Pi( \Theta_n^c|X_{1:n}^*)\mathbb{I}(G_n)]
        \le e^{(2+C_1)n\epsilon_n^2} \Pi(\Theta_n^c)\le e^{-2n\epsilon_n^2}\to0
    \end{align*}
by the second sieve assumption. On the other hand, for the first term, we need the following technical result on the bootstrap weight. Since $K_i\sim \text{Binom}(n, 1/n)$, by Markov's inequality
    \begin{align*}
        P(K_i> 2\log n)&\le e^{-2\log n}E(\exp(K_i))\\
        &=e^{-2\log n}(1-n^{-1} + n^{-1}e)^n\\
        &\le e^{-2\log n}e^{e-1}=e^{e-1}/n^2,
    \end{align*}
where the last inequality follows from the inequality $1+x\le e^{x}$ for any $x\in \mathbb{R}$. Thus, by the union bound, we have
\begin{align}
    \label{eq:bweight_bound}
        P(\max_{1\le i\le n}K_i>2 \log n)\le\frac{e^{e-1}}{n}
    \end{align}
which tends to zero as $n\to\infty$. Then, following the proof strategy of  \cite{han2019statistical}, we define
    \begin{align*}
        U_\theta(X)=\max\left\{\log \frac{p(X|\theta)}{p_0(X)},-\tau\right\}
    \end{align*}
for $\tau>0$.
Then we consider the weighted likelihood ratio empirical process defined as
    \begin{align*}
        \nu_n(\theta)=\frac{1}{\sqrt{n}}\sum_{i=1}^n [K_i U_\theta(X_i)-E(K_iU_\theta(X_i))]
    \end{align*}
Then by Lemma 12 of \cite{han2019statistical} together with the fact given in \labelcref{eq:bweight_bound} and our complexity assumption, we have
    \begin{align}
    \label{eq:bern_ineq}
      P\left(  \sup_{\theta\in \Theta_n: H(P_\theta, P_0)\le \sqrt{2}\epsilon}\nu_n(\theta)\ge \frac{k}{2}\sqrt{n}\epsilon^2\right)
\le 3\exp\left(-C_2'\frac{n\epsilon^2}{1+2\log n}\right)
    \end{align}
for any $\epsilon>\epsilon_n$ for some constants $C_2'>0$ and $k\in(1/2,1)$. By Lemma 4 of \cite{wong1995probability}, $E(\nu_n(\theta))\le -(1-\delta_0) H^2(P_\theta, P_0)$ with $\delta=2\exp(-\tau/2)/(1-\exp(-\tau/2))^2$. Let $\Theta_n(\epsilon;\sqrt{2}\epsilon)=\{\theta\in \Theta_n:\epsilon \le H(P_\theta, P_0)\le \sqrt{2}\epsilon \}$. Then we have
    \begin{align*}
       B(\epsilon;\sqrt{2}\epsilon) &=\left\{\sup_{\theta\in \Theta_n(\epsilon;\sqrt{2}\epsilon)}\frac{p_{\theta, K_{1:n}}(X_{1:n})}{p_{0, K_{1:n}}(X_{1:n})}\ge \exp(-n\epsilon^2(1-\delta_0-k/2))\right\}\\
       & \subset\left\{\sup_{\theta\in\Theta_n(\epsilon;\sqrt{2}\epsilon)}\nu_n(\theta)\ge \frac{k}{2}\sqrt{n}\epsilon^2\right\}
    \end{align*}
Hence, by \labelcref{eq:bern_ineq}, the probability of the event $B(\epsilon;\sqrt{2}\epsilon) $ is bounded above by $3\exp\left(-C_1'\frac{n\epsilon^2}{1+2\log n}\right).$ Let $J$ be the smallest integer such that $2^J\eta_n^2\ge 4$. Then by using a peeling technique, we have
    \begin{align*}
        &P\left(\sup_{\theta\in\Theta_n:H(P_\theta, P_0)\ge\eta_n }\frac{p_{\theta, K_{1:n}}(X_{1:n})}{p_{0, K_{1:n}}(X_{1:n})}\ge \exp(-n\eta_n^2(1-\delta_0-k/2)\right)\\
        &\le \sum_{j=0}^J P\left( B(\sqrt{2^{j}}\eta_n;\sqrt{2^{j+1}}\eta_n) \right)
        \le \sum_{j=0}^J 3\exp\left(-C_2'\frac{n2^j\eta_n^2}{1+2\log n}\right)\\
        &\le 4\exp\left(-C_2'\frac{n\eta_n^2}{1+2\log n}\right)
        \le 4\exp\left(-\frac{1}{3}C_2'M_n n\epsilon_n^2\right).
    \end{align*}
Therefore,  $E[ \Pi(A_n\cap \Theta_n|X_{1:n}^*)\mathbb{I}(G_n)]\le \exp(-C_3'n\eta_n^2)$ for some constant $C_3'>0$. Combining these derivations together, we have that by taking $t=C_4'n\eta_n^2$ with $C_4'$ being a positive constant less than $C_3'$, the term \labelcref{eq:vi_ineq} converges to zero, which proves that the variational posterior $Q^*$ with a single bootstrapped sample contracts to $P_0$ at a rate $\eta_n$.

\bibliographystyle{plainnat}
\bibliography{references, _references-LL_avb}

 \addcontentsline{toc}{section}{References}
\end{document}